\documentclass [a4paper,12pt]{article}
\usepackage{amsmath}
\usepackage{amsfonts}
\usepackage{indentfirst}
\usepackage{esint}
\usepackage{latexsym}
\usepackage{bbm}
\usepackage{amsfonts}
\usepackage{mathrsfs}
\usepackage{amssymb}
\usepackage{amsmath}
\usepackage{amsthm}
\usepackage{latexsym}
\usepackage{indentfirst}
\usepackage{enumitem}
\usepackage{bm}
\usepackage{esint}
\usepackage{hyperref}
\usepackage{cleveref}
\usepackage[numbers,square]{natbib}
\numberwithin{equation}{section}
\usepackage{color}
\allowdisplaybreaks
\usepackage{amsthm}
\usepackage{amssymb}
\usepackage{enumerate}
\usepackage{ulem}
\allowdisplaybreaks

\newtheorem{theorem}{Theorem}[section]
\newtheorem{lemma}[theorem]{Lemma}
\newtheorem{thm}[theorem]{Theorem}

\newtheorem{cor}[theorem]{Corollary}

\newtheorem{rmk}[theorem]{Remark}

\makeatletter
\usepackage{amsmath}
\usepackage{amsfonts}
\usepackage{indentfirst}
\usepackage{esint}
\usepackage{latexsym}
\usepackage{color}
\UseRawInputEncoding
\usepackage{amsthm}
\usepackage{amssymb}
\usepackage{enumerate}
\usepackage{ulem}
\makeatletter

\newcommand{\Rmnum}[1]{\expandafter\@slowromancap\romannumeral #1@}
\makeatother
\begin{document}
\title{
Parabolic Homogenization with an Interface}
\author{Yiping Zhang\thanks{Yanqi Lake Beijing Institute of Mathematical Sciences and Applications, Beijing 101408, China and  Yau Mathematical Sciences Center, Tsinghua University, Beijing 100084, China, (zhangyiping161@mails.ucas.ac.cn).}}
\date{}
\maketitle
\begin{abstract}
This paper considers a family of second-order parabolic equations in divergence form with rapidly
oscillating and time-dependent periodic coefficients and an interface between two periodic structures (i.e., composed of two different periodic media separated by the interface).
Following a framework initiated by Blanc, Le Bris and Lions and a generalized two-scale expansion in
divergence form of elliptic homogenization with an interface by Josien, we can determine the effective (or
homogenized) equation with the coefficient matrix being piecewise constant and discontinuous across the
interface. Moreover, we obtain the $O(\varepsilon)$ convergence rates in $L^{2(d+2)/{d}}_{x,t}$ (for $d\geq 2$) with $\varepsilon$-smoothing method and the uniform
interior Lipschitz estimates via compactness argument.

\end{abstract}

\section{Introduction and main results}
\let\thefootnote\relax\footnotetext{On behalf of all authors, the corresponding author states that there is no conflict of interest. Data sharing not applicable to this article as no datasets were generated or analyzed during the current study.}
Over the last forty years, there is a vast and rich mathematical literature on elliptic homogenization. Most
of these works are focused on qualitative results, such as proving the existence of a homogenized equation.
However, until recently, nearly all of the quantitative theory, such as the convergence rates in $L^2$ and
$H^1$, the $W^{1,p}$-estimates, the Lipschitz estimates and the asymptotic expansions of the Green functions
and fundamental solutions, were confined in periodic homogenization. There are many good expositions on this
topic, see for instance the books \cite{MR2839402,MR1329546,MR3838419}
for the periodic case, see also the book \cite{MR3932093} for the stochastic case. Moreover, the theory of
parabolic periodic homogenization is also well understood, and we refer to \cite{MR3361284,MR3596717,geng2020asymptotic,MR4072212} and their references therein for more results, see also \cite{geng2021homogenization,xu2020convergence} for the locally periodic coefficients in parabolic homogenization.

In this paper, we investigate the quantitative homogenization results of parabolic periodic equations in divergence
form with rapidly oscillating and time-dependent periodic coefficients and an interface between two periodic
structures, which is initiated by Blanc, Le Bris and Lions \cite{MR3421758} and further developed by Josien \cite{MR3974127} in elliptic
homogenization of divergence form. Since the gradient of the solution $u_0$ to the effective equation is
discontinuous across the interface, the main idea in \cite{MR3974127} is to generalize the two-scale expansion in periodic case without an
interface and recover the regularity of  $\nabla u_0$ after suitable transformation. Inspired by the generalized two-scale expansion in \cite{MR3974127}, we are able to investigate the problem of parabolic homogenization with an interface.

Besides the theory of the divergence form of elliptic equations in homogenization with an interface, Hairer and Manson \cite{MR2653267,MR2789509} have also considered the diffusion process (non-divergence form of elliptic equations) with coefficients that are
periodic outside of an interface region of finite thickness. They investigated the limiting long time/large
scale of such a process under diffusive rescaling and determined the generator of the limiting process, using
the probabilistic methods. Moreover, our another paper \cite{zhang2023quantitative} continues to study the problem above, and in that paper, we can
determine the effective equation, obtain the $O(\varepsilon)$ convergence rates and the uniform interior
Lipschitz estimates.

In this paper, for $0<\varepsilon\leq 1$ and $0<T<\infty$, we consider the following equations:
\begin{equation}
\left\{\begin{aligned}\partial_t u_\varepsilon -\operatorname{div}(A(x/\varepsilon,t/\varepsilon^2)\nabla u_\varepsilon)&=f(x,t)&&\text{in}\quad\mathbb{R}^d\times (0,T),\\
u_\varepsilon(x,0)&=0& &\text{for}\quad x\in\mathbb{R}^d.\end{aligned}\right.
\end{equation}

We will always assume that $A(y,s)=(A_{ij}(y,s))$ with $1\leq i, j\leq d$ is real, symmetric ($A^*=A$), and satisfies the following assumptions:\\

\noindent(i) Interface Assumption. The matrix $A(y,s)$ is of the following form:
\begin{equation}A(y,s)=
\left\{\begin{aligned}
A_+(y,s)\quad \text{if}\quad y_1>0,\\
A_-(y,s)\quad \text{if}\quad y_1<0,
\end{aligned}\right.\end{equation}
where $A_\pm$ are 1-periodic in $(y,s)$ and $A^*_{\pm}=A_\pm$. \\

\noindent(ii) Ellipticity Assumption.  The matrices $A_\pm$ satisfy the following ellipticity condition:
\begin{equation}
\kappa |\xi|^2\leq {A}_{\pm,ij}(y,s)\xi_i\xi_j \leq \kappa^{-1}|\xi|^2 \text{\quad for }(y,s)\in \mathbb{R}^{d+1} \text{ and }\xi\in \mathbb{R}^d,
\end{equation}where $\kappa\in (0,1)$ is a constant.\\

\noindent (iii) Smoothness Assumption. There exists a constant $M>0$ such that
\begin{equation}
|{A}_{\pm}(x,t)-{A}_{\pm}(x,s)|\leq M |t-s|,\end{equation} for any $(x,t,s)\in \mathbb{R}^{d+2}$.

From now on, the summation convention is used throughout the paper, and we also use the following notations:
$$\begin{aligned}
&d\geq 2 \text{ the space dimension};\  \mathbb{T}\cong\mathbb{R}/\mathbb{Z} \text{ the 1-D torus};\ \mathbb{D}=:\mathbb{R}\times \mathbb{T}^{d};\\
&\mathbb{D}_N=:(-N,N)\times \mathbb{T}^{d},\ \mathbb{D}_+=:(0,\infty)\times \mathbb{T}^d,\ \mathbb{D}_-=:(-\infty,0)\times \mathbb{T}^d;\\
&\Omega_T=:\mathbb{R}^d\times (0,T)\text{ for a positive } T\in (0,\infty);\\
&f(y,s) \text{ is } \mathbb{D}\text{-periodic if }f \text{ is 1-periodic in }(y',s) \text{ for }(y,s)=:(y_1,y',s);\\
&f^\varepsilon(x,t)=:f(x/\varepsilon,t/\varepsilon^2);\ \text{we use the subscripts }k,i,j,\ell\text{ if }1\leq k,i,j,\ell\leq d,\\
&\text{and the subscripts }K,I,J\text{ if }1\leq K,I,J\leq d+1 \text{ with } y_{d+1}=:s.\\
\end{aligned}$$
Under the ellipticity condition $(1.3)$, and for any $f\in L^2(0,T;H^{-1}(\mathbb{R}^d))$, it is well known that the Equation $(1.1)$ has a unique weak solution $u_\varepsilon$ in $L^\infty(0,T;L^2(\mathbb{R}^d))\cap L^2(0,T;\dot{H}^{1}(\mathbb{R}^d))$, satisfying the following energy estimates
\begin{equation*}
||u_\varepsilon||_{L^\infty(0,T;L^2(\mathbb{R}^d))}+||\nabla u_\varepsilon||_{L^2(\Omega_T)}\leq C||f||_{L^2(0,T;H^{-1}(\mathbb{R}^d))}.
\end{equation*}

\noindent Moreover, by embedding inequality, there holds
\begin{equation}
||u_\varepsilon||_{L^{\frac{2(d+2)}{d}}(\Omega_T)}\leq C||u_\varepsilon||_{L^\infty(0,T;L^2(\mathbb{R}^d))\cap L^2(0,T;\dot{H}^{1}(\mathbb{R}^d)) }\leq C||f||_{L^2(0,T;H^{-1}(\mathbb{R}^d))}.
\end{equation}

\noindent Similarly, if $f\in L^2(0,T;L^2(\mathbb{R}^d))$, then $u_\varepsilon$ satisfies the following estimates
\begin{equation}
||u_\varepsilon||_{L^\infty(0,T;L^2(\mathbb{R}^d))}+||\nabla u_\varepsilon||_{L^2(\Omega_T)}\leq CT^{1/2}||f||_{L^2(0,T;L^2(\mathbb{R}^d))}.
\end{equation}

Inspired by the classical periodic parabolic homogenization theory by Geng and Shen \cite{geng2020asymptotic,MR3361284,MR3596717,MR4072212} and the work considering the elliptic homogenization with an interface by Josien \cite{MR3974127}, one may expect the homogenized equation of $(1.1)$ is given by:
\begin{equation}\left\{\begin{aligned}
\partial_t u_0 -\operatorname{div}(\widehat{A}\nabla u_0)&=f(x,t)&\quad &\text{in}\quad\mathbb{R}^d\times (0,T),\\
u_0(x,0)&=0& \quad &\text{for}\quad x\in\mathbb{R}^d,
\end{aligned}\right.\end{equation}where the homogenized coefficient matrix $\widehat{A}$ is defined by
\begin{equation}\widehat{A}(y)=
\left\{\begin{aligned}
\widehat{A_+}\quad \text{if}\quad y_1>0,\\
\widehat{A_-}\quad \text{if}\quad y_1<0,
\end{aligned}\right.\end{equation}
and $\widehat{A_\pm}$ are the homogenized matrices associated with the periodic matrices $A_\pm$ (see Section 2 for the details).

In general, the matrix $\widehat{A}$ is discontinuous across the interface $\mathcal{I}=:\{y|y_1=0\}$. However, if there is no interface assumption in the mathematical setting, then the matrix $\widehat{A}$ is a constant matrix satisfying the ellipticity
condition $(1.3)$, and $u_0$ enjoys better regularity estimates than $u_\varepsilon$. Therefore, we need first to recover
the regularity of $u_0$ after suitable transformation across the interface $\mathcal{I}$, which is the main idea to
introduce the concept of $\widehat{A}$-harmonic functions $P_j(x)$, see Section 2 for more details.

After introducing the $\widehat{A}$-harmonic functions $P_j(x)$, we need to give the  definition of correctors suitably, and to investigate the properties of these correctors, especially for the decay estimates, which are used to prove the $H$-convergence and to introduce the so-called dual correctors \cite[Lemma 2.1]{MR3596717}.

After obtaining the suitable properties of these correctors, we are ready to obtain the
$O(\varepsilon)$-convergence rates, using the so-called $\varepsilon$-smoothing method. As pointed out early, we need
to recover the regularity of $u_0$ across the interface with suitable transformation, which is also piecewise
constant across the interface. Unfortunately, as a consequence of this transformation, some similar and beautiful
results do not hold anymore, see Remark 3.6 for the details. Therefore, we need more regularity assumptions on the source term $f$ and the coefficient matrix $A$ to
obtain the $O(\varepsilon)$-convergence rates, which may be released by another useful observations and different methods.

Moreover, we also investigate the interior Lipschitz estimates via compactness argument \cite{MR910954,MR978702}. Due to the classical parabolic theory, it is easy to
see that if $\widehat{A}$ is a constant matrix satisfying the ellipticity condition $(1.3)$, then $u_0$ enjoys better regularity for $f$ sufficiently regular,
which implies that, due to compactness, $u_\varepsilon$ inherits the medium-scale regularity property of the solution $u_0$. Back to
our setting with an interface, due to the $H$-convergence, we need also to find a suitable regularity of $u_0$ that $u_\varepsilon$ can
inherit for small $\varepsilon$. Then, following the ideas of \cite[Theorem 5.1]{MR3361284} and \cite[Theorem 4.1]{MR3974127}, we can also obtain the
uniform interior Lipschitz estimates.

Precisely, we have obtained the following results. The first one concerns the uniform $H$-convergence:
\begin{thm}[uniform $H$-convergence]
Let $\Omega$ be a bounded Lipschitz domain in $\mathbb{R}^d$, and  let $f\in L^2(T_0,T_1;W^{-1,2}(\Omega))$ and sequences $(x_k,t_k)\in \mathbb{R}^{d+1}$ and $\varepsilon_k\in \mathbb{R}^*_+$ satisfy $x_k\cdot e_1\rightarrow x_{0,1}\in \mathbb{R}$ and $\varepsilon_k\rightarrow 0$. Moreover, let $u_k\in L^2(T_0,T_1;W^{1,2}(\Omega))$ be a weak solution of
\begin{equation}
\partial_t u_k -\operatorname{div}\left[A_k((x-x_k)/\varepsilon_k,(t-t_k)/\varepsilon_k^2)\nabla u_k\right]=f\quad \text{in }\Omega\times (T_0,T_1),
\end{equation}
where the matrix $A_k(y,s)$ satisfies $(1.2)$ and $(1.3)$. Suppose that $\widehat{A_k}\rightarrow A^0$, with $\widehat{A_k}$ defined in $(1.8)$, and
$$\left\{\begin{aligned}
&u_k\rightharpoonup u\quad \quad\text{ weakly in }L^2(T_0,T_1;L^2(\Omega)),\\
&\nabla u_k\rightharpoonup \nabla u\quad \text{ weakly in }L^2(T_0,T_1;L^2(\Omega)).
\end{aligned}\right.$$
Then, $u\in L^2(T_0,T_1;W^{1,2}(\Omega))$ is a weak solution of
$$\partial_t u-\operatorname{div}\left({A^0}(x-x_{0,1}e_1)\nabla u\right)=f \quad \text{ in }\Omega\times (T_0,T_1),$$
where $e_1=(1,0,\cdots,0)$ and the matrix ${A^0}$ is piecewise constant and satisfies the ellipticity condition $(1.3)$ with some different $\kappa^{-1}$.
\end{thm}

The next one concerns the $O(\varepsilon)$-convergence rates in $L^{\frac{2(d+2)}{d}}(\Omega_T)$:
\begin{thm}[convergence rates] Suppose that the matrix $A$ satisfies $(1.2)$-$(1.4)$. Let $f,\partial_t f\in L^2(\Omega_T)$, and $u_\varepsilon$, $u_0$ be the weak solution to the Equations $(1.1)$ and $(1.7)$, respectively. Then there hold  the following convergence rates:
\begin{equation*}\begin{aligned}
||u_\varepsilon-u_0||_{L^{\frac{2(d+2)}{d}}(\Omega_T)}\leq C\varepsilon ||f||_{L^2(\Omega_T)\cap L^\infty (0,T;L^2(\mathbb{R}^d))}+C\varepsilon^2||f_t||_{L^2(\Omega_T)},
\end{aligned}\end{equation*}
where the constant $C$ depends only on $d$, $M$, $\kappa$ and $T$.
\end{thm}

\begin{rmk}Compared with \cite[Theorem 1.1]{MR3596717}, only the regularity assumption $f\in L^2$ is needed to guarantee the $O(\varepsilon)$ convergence rates. As pointed out early, some similar
results do not hold anymore after  recovering the regularity for $u_0$ across the interface. Then we need more regularity assumptions on the source term $f$ and the coefficient matrix $A$ to
obtain the $O(\varepsilon)$ convergence rates, which may be released. See Remark 3.6 for the details.

Meanwhile, one is inclined to think that having time-dependent coefficients in the parabolic equation may be not very physically relevant. Therefore, the condition $(1.4)$ is naturally satisfied if the coefficients are  time-independent.\end{rmk}

The last one concerns the uniform interior Lipschitz estimates:

\begin{thm}[interior Lipschitz estimates] Let $A=A(y,s)$ be a matrix satisfying $(1.2)$-$(1.3)$ and $A_{\pm}\in C^{\alpha,\alpha/2}(\mathbb{R}^{d+1})$ for some $\alpha>0$. Let $u_\varepsilon\in L^2(t_0-r^2,t_0;H^1(B(x_0,r)))$ be a weak solution of $\partial_t u_\varepsilon-\operatorname{div}(A^\varepsilon \nabla u_\varepsilon)=f$ in $2Q$ with $Q=Q_r(x_0,t_0)=:B(x_0,r)\times (t_0-r^2,t_0)$ and $f\in L^p(2Q)$ for some $p>d+2$. Then
\begin{equation}
||u_\varepsilon||_{C^{1,1/2}}(Q)\leq \frac{C}{r}\left(\fint_{2Q}|u_\varepsilon|^2\right)^{1/2}+Cr\left(\fint_{2Q}|f|^p\right)^{1/p},
\end{equation}
where the constant $C$ depends only on $d$, $\kappa$, $\alpha$ and $p$.
\end{thm}
In Theorem 1.4, we have used the notations
\begin{equation}
\begin{aligned}
\fint_{Q} u &=\frac{1}{|Q|} \int_{Q} u;\quad \|u\|_{C^{\alpha, \alpha / 2}(Q)} &=\sup _{\substack{(x, t),(y, s) \in Q \\
(x, t) \neq(y, s)}} \frac{|u(x, t)-u(y, s)|}{\left(|x-y|+|t-s|^{1 / 2}\right)^{\alpha}},0<\alpha \leq 1.
\end{aligned}
\end{equation}

\section{Preliminaries}
\subsection{$\widehat{A}$-harmonic functions}
Inspired by the works \cite{MR3421758,MR3974127}, we introduce the $\widehat{A}$-harmonic functions, which is spanned by the constants functions and the following piecewise linear functions:
\begin{equation}P_j(x)=P(x)\cdot e_j=:\left\{\begin{aligned}&x_j,&\text{ if } x_1<0,\\
&x_j+\vartheta_jx_1, &\text{ if } x_1>0,
\end{aligned}\right.\end{equation}
for $j=1,\cdots,d$, where $e_j=(0,\cdots,0,1,0,\cdots,0)$ with 1 in the j-th position and  $\vartheta=(\vartheta_1,\cdots,\vartheta_j)$ is related to the transmission matrix $\widehat{A}$ across the interface $\mathcal{I}$ and defined as:
\begin{equation}
\vartheta_j=\frac{\left(\widehat{A_-}\right)_{1j}-\left(\widehat{A_+}\right)_{1j}}{\left(\widehat{A_+}\right)_{11}}.
\end{equation}

If $\vartheta=0$, then $\widehat{A}$ is constant and the functions $P_j$ are linear.

It is straightforward to check that the functions $P_j$ are solutions to
\begin{equation}\operatorname{div}(\widehat{A}\nabla P_j)=0\quad \text{in}\quad \mathbb{R}^d.\end{equation}

Actually, by definition, the functions $P_j$ are Lipschitz continuous and their gradients read as
\begin{equation*}\nabla P_j(x)=\left\{\begin{aligned}
&e_j, &\text{ if } x_1<0,\\
&e_j+\vartheta_je_1, &\text{ if } x_1>0.
\end{aligned}\right.\end{equation*}

Hence, the functions $P_j$ are $\widehat{A}$-harmonic in $\mathbb{R}^*_-\times \mathbb{R}^{d-1}$ and in $\mathbb{R}^*_+\times \mathbb{R}^{d-1}$, and they satisfy the transmission conditions across the interface:
\begin{equation}
\begin{gathered}
\lim _{h \rightarrow 0^{+}}\left[\left(\widehat{A} \cdot \nabla P_{j}\right)\left(x+h e_{1}\right)\right] \cdot e_{1}=\lim _{h \rightarrow 0^{+}}\left[\left(\widehat{A} \cdot \nabla P_{j}\right)\left(x-h e_{1}\right)\right] \cdot e_{1}, \\
\lim _{h \rightarrow 0^{+}} \partial_{k} P_{j}\left(x+h e_{1}\right)=\lim _{h \rightarrow 0^{+}} \partial_{k} P_{j}\left(x-h e_{1}\right),
\end{gathered}
\end{equation}
for all $x\in\mathcal{I}$ and $k=2,\cdots,d$.

\subsection{Basic estimates of the correctors}
Since the correctors are meant to turn the $\widehat{A}$-harmonic functions $P_j$ into $A$-harmonic sub-linear functions, the correctors $\chi_k$  should solve the following equation:
\begin{equation}
\partial_s\chi_k-\partial_{y_i}\left(A_{ij}(y,s)\partial_{y_j}(P_k(y)+\chi_k)\right)=0 \text{ in }\mathbb{D},\ \chi_k\text{ is }\mathbb{D}\text{-periodic},
\end{equation}
for $k=1,2,\cdots, d$.

Recall that for the periodic matrices $A_\pm$, the corrector $\chi_{\pm,j}$ are defined as
the weak solution to the following cell problem:
\begin{equation}\left\{\begin{array}{l}
\left(\partial_{s}+\mathcal{L}_{\pm,1}\right)\left(\chi_{\pm,j}\right)=-\mathcal{L}_{\pm,1}\left(y_j\right) \quad \text { in }\quad \mathbb{T}^{d+1}, \\
\chi_{\pm,j}(y, s) \text { is } 1 \text {-periodic in }(y, s), \ \
\int_{\mathbb{T}^{d+1}} \chi_{\pm,j}=0.
\end{array}\right.\end{equation}
where $j=1,\cdots,d$, $\mathcal{L}_{\pm,\varepsilon}=-\operatorname{div}(A_{\pm}^\varepsilon\cdot \nabla)$ and $\mathcal{L}_{\varepsilon}=-\operatorname{div}(A^\varepsilon\cdot \nabla)$. Then the effective operator $\widehat{A_{\pm}}$ \cite{MR3596717} is given by
$$\widehat{A_{\pm}}=\fint_{\mathbb{T}^{d+1}}\left(A_{\pm}+A_{\pm}\nabla \chi_{\pm}\right).$$

For the correctors and dual correctors in parabolic homogenization without an interface, the following results hold true.
\begin{lemma}
Let $1\leq j\leq d$ and assume that $A_{\pm}$ are 1-periodic matrices and satisfy the ellipticity condition $(1.3)$. Then there exist 1-periodic functions $\phi_{\pm,KIj}(y,s)$ in $\mathbb{R}^{d+1}$ such that $\phi_{\pm,KIj}\in H^1(\mathbb{T}^{d+1})$,
\begin{equation}
B_{\pm,Ij}=\partial_{y_K}\phi_{\pm,KIj}\quad\text{and}\quad \phi_{\pm,KIj}=-\phi_{\pm,IKj},
\end{equation}
where $1\leq K, I\leq d+1$, $B_{\pm,ij}=A_{\pm,ij}+A_{\pm,ik}\partial_{y_k}\chi_{\pm,j}-\widehat{A}_{\pm,ij}$  for $i,j=1,\cdots,d$ and $B_{\pm,(d+1)j}=-\chi_{\pm,j}$, and we use the notation $y_{d+1}=:s$. Moreover, if $\partial_s A_{\pm}\in L^\infty(\mathbb{R}^{d+1})$, then $\partial_s\phi_{\pm}\in H^1(\mathbb{T}^{d+1})$.
\end{lemma}

\begin{proof}The proof is due to \cite[Lemma 2.1]{MR3596717}, and we point out the main ideas for simplicity. Due to $\int_{\mathbb{T}^{d+1}}B_{\pm,Ij}$=0, there exist $f_{\pm,Ij}\in H^2(\mathbb{T}^{d+1})$ solving
\begin{equation}\left\{\begin{aligned}
&\Delta_{d+1}f_{\pm,Ij}=B_{\pm,Ij} \quad \text{ in }\mathbb{T}^{d+1},\\
&f_{\pm,Ij}\text{ is 1-periodic in }(y,s), \fint_{\mathbb{T}^{d+1}}f_{\pm,Ij}=0,
\end{aligned}\right.\end{equation}
where $\Delta_{d+1}$ denotes the Laplacian in $\mathbb{R}^{d+1}$. Then  $\sum_{I=1}^{d+1}\partial_{y_I}f_{\pm,Ij}$ is a constant since it is a periodic harmonic function in $\mathbb{R}^{d+1}$.  By setting
\begin{equation*}\phi_{\pm,KIj}=\partial_{y_K}f_{\pm,Ij}-\partial_{y_I}f_{\pm,Kj},\end{equation*}
we complete the proof of the first part.

Moreover, if $\partial_s A_{\pm}\in L^\infty(\mathbb{R}^{d+1})$, then $\partial_s \nabla_y \chi_{\pm}\in L^2(\mathbb{T}^{d+1})$ and $\partial_s \chi_{\pm}\in L^2(\mathbb{T}^{d+1})$ due to Equation $(2.6)$ and hence $\partial_s B_{\pm}\in L^2(\mathbb{T}^{d+1})$. Consequently, by regularity estimates, $\partial_s f_{\pm}\in H^2(\mathbb{T}^{d+1})$ and hence $\partial_s\phi_{\pm}\in H^1(\mathbb{T}^{d+1})$.
\end{proof}

\begin{thm}Under the assumptions $(1.2)$ and $(1.3)$,
there exists a unique (up to the addition of a constant) $\mathbb{D}$-periodic corrector $\chi_j$ to the Equation $(2.5)$, satisfying
\begin{equation}\left\{\begin{aligned}
&\nabla (\chi_j-\chi_{-,j})\in L^2(\mathbb{D}_-),\\
&\nabla (\chi_j-\chi_{+,j}-\vartheta_j\chi_{+,1})\in L^2(\mathbb{D}_+).
\end{aligned}\right.\end{equation}
\end{thm}
\begin{proof}
The idea is due to \cite{MR3421758,MR3974127} for the elliptic case. Choose smoothing cut-off functions $\psi_\pm(y_1)$, such that
\begin{equation}\begin{aligned}
&\psi_+(y_1)=1,\text{ if }y_1\geq 1;\ \psi_+(y_1)=0,\text{ if }y_1\leq 0;\\
&\psi_-(y_1)=1,\text{ if }y_1\leq -1;\ \psi_-(y_1)=0,\text{ if }y_1\geq 0.
\end{aligned}\end{equation}

Let $v=:\chi_j-\psi_+(\chi_{+,j}+\vartheta_j\chi_{+,1})-\psi_-\chi_{-,j}$, then a direct computation shows that
\begin{equation}\begin{aligned}
\partial_s v-\operatorname{div}(A\nabla v)=&\operatorname{div}(A\nabla P_j)-\partial_s\left(\psi_+(\chi_{+,j}+\vartheta_j\chi_{+,1})\right)-\partial_s(\psi_-\chi_{-,j})\\
&+\operatorname{div}\left[A\nabla\psi_+(\chi_{+,j}+\vartheta_j\chi_{+,1})+A\nabla\psi_-\chi_{-,j}\right.\\
&\quad\quad\quad\left.+A\psi_+(\nabla\chi_{+,j}+\vartheta_j\nabla\chi_{+,1})+A \psi_-\nabla\chi_{-,j}\right]\\
=&:\operatorname{div}f+g,
\end{aligned}\end{equation}
where by adding the constant term $\widehat{A}\nabla P_j$, we have
\begin{equation}f=(1-\psi_+-\psi_-)(A-\widehat{A})\nabla P_j+A\nabla\psi_+(\chi_{+,j}+\vartheta_j\chi_{+,1})+A\nabla\psi_-\chi_{-,j}\end{equation}
and using $(2.3)$ yields that
$$\begin{aligned}g=:&-\partial_s\left(\psi_+(\chi_{+,j}+\vartheta_j\chi_{+,1})\right)-\partial_s(\psi_-\chi_{-,j})
+\operatorname{div}\left((\psi_++\psi_-)(A-\widehat{A})\nabla P_j\right)\\
&+\operatorname{div}\left[A\psi_+(\nabla\chi_{+,j}+\vartheta_j\nabla\chi_{+,1})+A \psi_-\nabla\chi_{-,j}\right]\\
=:&\ g_++g_-.
\end{aligned}$$

Before moving forward, we need to rewrite $g$ in a more suitable form. For simplicity, we only perform the computation on $g_+$. In view of Lemma 2.1 and using the notation $y_{d+1}=s$, we have
\begin{equation}\begin{aligned}
-\partial_s\left(\psi_+(\chi_{+,j}+\vartheta_j\chi_{+,1})\right)
=&-\sum_{\ell=1}^{d}\partial_{d+1}\left(\psi_+\chi_{+,\ell}\partial_\ell P_j\right)=\sum_{\ell=1}^{d}\partial_{d+1}\left(\psi_+B_{+,(d+1)\ell}\partial_\ell P_j\right)\\
=&\sum_{\ell=1}^{d}\sum_{K=1}^{d+1}\partial_{d+1}\left(\psi_+\partial_K\phi_{+,K(d+1)\ell}\partial_\ell P_j\right),
\end{aligned}\end{equation}

and \begin{equation}\begin{aligned}
\operatorname{div}&\left[\psi_+A(\nabla\chi_{+,j}+\vartheta_j\nabla\chi_{+,1})+\psi_+(A-\widehat{A})\nabla P_j\right]\\
&=\sum_{i,\ell=1}^d\partial_i\left[\psi_+\left(A_{i\ell}+A_{ik}\partial_k\chi_{+,\ell}-\widehat{A}_{+,i\ell}\right)\partial_\ell P_j\right]\\
&=\sum_{i,\ell=1}^d\sum_{K=1}^{d+1}\partial_i\left(\psi_+\partial_K\phi_{+,Ki\ell}\partial_\ell P_j\right).
\end{aligned}\end{equation}
Note that $\nabla P_j$ is constant everywhere but on the interface $\mathcal{I}$ where $\psi_{\pm}$ vanish.
Therefore, combining $(2.13)$-$(2.14)$ after using Lemma 2.1 yields that
\begin{equation*}\begin{aligned}
g_+&=\sum_{\ell=1}^d\sum_{I,K=1}^{d+1}\partial_I\left(\psi_+\partial_K\phi_{+,KI\ell}\partial_\ell P_j\right)
=\sum_{\ell=1}^d\sum_{I,K=1}^{d+1}\partial_I\psi_+\partial_K\phi_{+,KI\ell}\partial_\ell P_j\\
&=\sum_{\ell=1}^d\sum_{I,K=1}^{d+1}\partial_K\left(\partial_I\psi_+\phi_{+,KI\ell}\partial_\ell P_j\right)=\sum_{\ell=1}^d\sum_{K=1}^{d+1}\partial_K\left(\partial_1\psi_+\phi_{+,K1\ell}\partial_\ell P_j\right)\\
&=\sum_{\ell,k=1}^d\partial_k\left(\partial_1\psi_+\phi_{+,k1\ell}\partial_\ell P_j\right)+\sum_{\ell=1}^d \partial_1\psi_+\partial_s\phi_{+,(d+1)1\ell}\partial_\ell P_j.
\end{aligned}\end{equation*}

To proceed, let $w=:v-\partial_1\psi_+\phi_{+,(d+1)1\ell}\partial_\ell P_j-\partial_1\psi_-\phi_{-,(d+1)1\ell}\partial_\ell P_j$, then
\begin{equation}\begin{aligned}
\partial_s w-\operatorname{div}(A\nabla w)=&\operatorname{div}f+\partial_k\left(\partial_1\psi_+\phi_{+,k1\ell}\partial_\ell P_j\right)+\partial_k\left(\partial_1\psi_-\phi_{-,k1\ell}\partial_\ell P_j\right)\\
&+\operatorname{div}\left[A\nabla (\partial_1\psi_+\phi_{+,(d+1)1\ell}\partial_\ell P_j+\partial_1\psi_-\phi_{-,(d+1)1\ell}\partial_\ell P_j)\right]\\
=&:\operatorname{div}\tilde{f}.
\end{aligned}\end{equation}

Since $\text{supp}(\tilde{f})\subset \mathbb{D}_1$ and  $\tilde{f}\in L^2(\mathbb{D})$, then there exists a unique (up to the addition of a constant) weak solution $w$, satisfying $\nabla w\in L^2 (\mathbb{D})$, to the Equation $(2.15)$, with the energy estimate
\begin{equation}||\nabla w||_{L^2 (\mathbb{D})}\leq C ||\tilde{f}||_{L^2(\mathbb{D})}\leq C(d,\kappa),\end{equation}
which immediately  implies that there exists a unique (up to the addition of a constant) $\mathbb{D}$-periodic corrector $\chi_j$ to the Equation $(2.5)$, satisfying the energy estimates $(2.9)$.
Consequently, we complete this proof.
\end{proof}

\begin{rmk}
To solve the Equation $(2.15)$, we first consider the following equations: for $R>1$, we consider
\begin{equation}\left\{\begin{aligned}
&\partial_s w_R-\operatorname{div}(A\nabla w_R)=\operatorname{div}\tilde{f}\quad &\text{in}\quad (y,s)\in \mathbb{D}_R,\\
&w_R(-R,y',s)=w_R(R,y',s)=0 \quad &\text{for}\quad (y',s)\in \mathbb{T}^d.
\end{aligned}\right.\end{equation}
Then the Galarkin's method would ensure a unique $\mathbb{D}$-periodic weak solution $w_R$ to the equation above, satisfying the energy estimates:
$$||\partial_sw_R||_{L^2(\mathbb{T};H^{-1}((-R,R)\times \mathbb{T}^{d-1}))}+||\nabla w_R||_{L^2(\mathbb{D}_R)}\leq C ||\tilde{f}||_{L^2(\mathbb{D})},$$
with the constant $C$ independent of $R$,  which would imply that, up to a
diagonal subsequence that we still denote by $w_R$, for any $G\subset\subset \mathbb{R}$, we have the following weak convergence:
\begin{equation*}\begin{aligned}
&\partial_s w_R\rightharpoonup \partial_s w \ \text{weakly in }L^2(\mathbb{T};H^{-1}(G\times \mathbb{T}^{d-1})),\text{ as }R\rightarrow\infty,\\
&\nabla w_R\rightharpoonup \nabla w \ \text{weakly in }L^2(G\times \mathbb{T}^{d}),\text{ as }R\rightarrow\infty.
\end{aligned}\end{equation*}
Thus we obtain a unique (up to the addition of a constant) weak solution $w$ to the Equation $(2.15)$, satisfying the energy estimate $(2.16)$.

Moreover, if $\partial_s A\in L^\infty(\mathbb{D})$, then $\partial_s \tilde{f}\in L^2(\mathbb{D})$ due to Lemma 2.1, $(2.12)$ as well as $(2.15)$, and $\partial_s w_R$ satisfies
\begin{equation*}\left\{\begin{aligned}
&\partial_s (\partial_sw_R)-\operatorname{div}(A\nabla \partial_sw_R)=\operatorname{div}(\partial_s\tilde{f})+\operatorname{div}(\partial_sA\nabla w_R)\quad &\text{in}\quad (y,s)\in \mathbb{D}_R,\\
&\partial_s w_R(-R,y',s)=\partial_s w_R(R,y',s)=0 \quad &\text{for}\quad (y',s)\in \mathbb{T}^d.
\end{aligned}\right.\end{equation*}

Energy estimate yields that
$$\begin{aligned}||\nabla \partial_s w_R||_{L^2(\mathbb{D}_R)}
&\leq C ||\partial_s\tilde{f}||_{L^2(\mathbb{D})}+C ||\nabla w_R||_{L^2(\mathbb{D})}\\
&\leq C ||\partial_s\tilde{f}||_{L^2(\mathbb{D})}+C||\tilde{f}||_{L^2(\mathbb{D})}.
\end{aligned}$$

Next, multiplying the Equation $(2.17)$ by $\partial_s w_R$ and integrating the resulting equation over $\mathbb{D}_R$ yields that
$$\begin{aligned}
\|\partial_s w_R\|_{L^2(\mathbb{D}_R)}^2
&\leq C ||\nabla w_R||_{L^2(\mathbb{D}_R)}||\nabla \partial_s w_R||_{L^2(\mathbb{D}_R)}+C||\tilde{f}||_{L^2(\mathbb{D}_R)} ||\nabla \partial_s w_R||_{L^2(\mathbb{D}_R)}\\
&\leq C ||\tilde{f}||_{L^2(\mathbb{D})}^2+C||\partial_s \tilde{f}||_{L^2(\mathbb{D})}^2.
\end{aligned}$$

Consequently, from the computation above, if $\partial_s A\in L^\infty(\mathbb{D})$, then we have the following additional energy estimates to the Equation $(2.15)$:
\begin{equation}
\|\partial_s w\|_{L^2(\mathbb{D})}+||\nabla \partial_s w||_{L^2(\mathbb{D})}\leq C ||\tilde{f}||_{L^2(\mathbb{D})}+C||\partial_s \tilde{f}||_{L^2(\mathbb{D})}.
\end{equation}

Moreover, form the construction above and $\partial_s A\in L^\infty(\mathbb{D})$, we can also obtain the far field behaviours:
\begin{equation}w(y_1,y',s)\rightarrow 0,\text{ for any }(y',s)\in \mathbb{T}^d,\quad \text{as }|y_1|\rightarrow\infty,\end{equation}
and the solution $w$ is unique in the sense of the far field behaviours $(2.19)$. Since the proof of $(2.19)$ is totally similar to the proof of Theorem 2.4, we postpone it until Remark 2.5.
\end{rmk}

Now we are ready to prove Theorem 1.1, which is classical and similar to \cite[Theorem 3.6]{MR3361284}. Therefore, we only emphasize on its main ingredient.\\

\noindent\textbf{Proof of Theorem 1.1}: For simplicity, we assume $x_k=0$ and $t_k=0$. Let $P^k$ be the piecewise continuous functions associated with $\widehat{A_k}$ defined in $(2.1)$, for $k=1,2,\cdots$, and $P^0$ associated with $A^0$, respectively. Note that $\widehat{A_k}\rightarrow A^0$, then $\nabla P^k\rightarrow \nabla P^0$.
Moreover, we denote $\Omega_+=:\{x|x\in \Omega,x_1>0\}$ and similarly for $\Omega_-$. Then there holds

\begin{equation}\begin{aligned}
A_k^{\varepsilon_k}\nabla_y \chi_k^{\varepsilon_k}+A_k^{\varepsilon_k}\nabla P^k(x)&=A_{-,k}^{\varepsilon_k}\nabla_y \chi_k^{\varepsilon_k}+A_{-,k}^{\varepsilon_k}\nabla P^k(x)\\
&=A_{-,k}^{\varepsilon_k}\left(\nabla_y \chi_k^{\varepsilon_k}-\nabla_y \chi_{-,k}^{\varepsilon_k}\right)+A_{-,k}^{\varepsilon_k}\nabla P^k(x)+A_{-,k}^{\varepsilon_k}\nabla_y \chi_{-,k}^{\varepsilon_k}\\
&\rightharpoonup \lim_{k\rightarrow\infty}\int_{\mathbb{T}^{d+1}}\left(A_{-,k}^{\varepsilon_k}\nabla P^k(x)+A_{-,k}^{\varepsilon_k}\nabla_y \chi_{-,k}^{\varepsilon_k}\right)\\
&=\lim_{k\rightarrow\infty}\widehat{A_{-,k}}\nabla P^k(x)\\
&=\widehat{A^{-,0}}\nabla P^0(x) \text{ weakly in }L^2(\Omega_-\times (T_0,T_1)),
\end{aligned}\end{equation}
due to \cite[Lemma 3.4]{MR3361284} and
$$||\nabla_y \chi_k^{\varepsilon_k}-\nabla_y \chi_{-,k}^{\varepsilon_k}||_{L^2(\Omega_-\times (T_0,T_1))}\leq C \varepsilon_k^{1/2} ||\nabla_y \chi_k-\nabla_y \chi_{-,k}||_{L^2(\mathbb{R}_-\times \mathbb{T}^d)}\leq C\varepsilon_k^{1/2}.$$

Similarly,
\begin{equation}\begin{aligned}
&A_k^{\varepsilon_k}\nabla_y \chi_k^{\varepsilon_k}+A_k^{\varepsilon_k}\nabla P^k(x)\\
=&A_{+,k}^{\varepsilon_k}\nabla_y \chi_k^{\varepsilon_k}+A_{+,k}^{\varepsilon_k}\nabla P^k(x)\\
=&A_{+,k}^{\varepsilon_k}\left(\nabla_y \chi_k^{\varepsilon_k}-\nabla_y \chi_{+,k}^{\varepsilon_k}\cdot \nabla P^k\right)+A_{+,k}^{\varepsilon_k}\nabla P^k(x)+A_{+,k}^{\varepsilon_k}\nabla_y \chi_{+,k}^{\varepsilon_k}\cdot \nabla P^k\\
\rightharpoonup& \lim_{k\rightarrow\infty}\int_{\mathbb{T}^{d+1}}\left(A_{+,k}^{\varepsilon_k}\nabla P^k(x)+A_{+,k}^{\varepsilon_k}\nabla_y \chi_{+,k}^{\varepsilon_k}\cdot \nabla P^k\right)\\
=&\lim_{k\rightarrow\infty}\widehat{A_{+,k}}\nabla P^k(x)\\
=&\widehat{A^{+,0}}\nabla P^0(x) \text{ weakly in }L^2(\Omega_+\times (T_0,T_1)),
\end{aligned}\end{equation}
due to \cite[Lemma 3.4]{MR3361284} and
$$||\nabla_y \chi_k^{\varepsilon_k}-\nabla_y \chi_{+,k}^{\varepsilon_k}\nabla P^k||_{L^2(\Omega_+\times (T_0,T_1))}\leq C \varepsilon_k^{1/2} ||\nabla_y \chi_k-\nabla_y \chi_{+,k}\nabla P^k||_{L^2(\mathbb{R}_+\times \mathbb{T}^d)}\leq C\varepsilon_k^{1/2}.$$

Therefore, combining $(2.20)$-$(2.21)$ yields that
$$A_k^{\varepsilon_k}\nabla_y \chi_k^{\varepsilon_k}+A_k^{\varepsilon_k}\nabla P^k\rightharpoonup \widehat{A^{0}}\nabla P^0 \text{ weakly in }L^2(\Omega\times (T_0,T_1)).$$

Meanwhile, note that $\operatorname{div}\left(\widehat{A^{0}}\nabla P_j^0\right)=0$ for $j=1,\cdots,d$, then for any $\varphi \in C^1_0(\Omega\times (T_0,T_1))$, we have
$$\int_{T_0}^{T_1}\int_{\Omega}\widehat{A^{0}}\nabla P^0\cdot u \nabla \varphi=-\int_{T_0}^{T_1}\int_{\Omega}\widehat{A^{0}}\nabla P^0\cdot \varphi\nabla u.$$

In fact, the equality above is useful to determine the effective equation. After obtaining the facts above, the other computation is totally similar to \cite[Theorem 3.6]{MR3361284}, which we omit and refer to \cite[Theorem 3.6]{MR3361284} for more details.\qed\\

Next, in order to obtain the so-called dual correctors \cite[Lemma 2.1]{MR3596717}, we need to investigate the decay estimates for $w$ defined in $(2.15)$.
\begin{thm}[decay estimates] Assume that the matrix $A$ satisfies the assumptions $(1.2)$-$(1.4)$, then there exists some positive constant $\lambda$, depending only on $d$, $\kappa$ and $M$, such that
\begin{equation}\begin{aligned}
&\int_{(-\infty,y_1)\times \mathbb{T}^{d}}|\nabla_{y,s}\chi_j-\nabla_{y,s}\chi_{-,j}|^2\leq C\exp(-\lambda|y_1|)\quad &\text{if } y_1<-1;\\
&\int_{(y_1,\infty)\times \mathbb{T}^{d}}|\nabla_{y,s}\chi_j-\nabla_{y,s}\chi_{+,j}-\vartheta_j\nabla_{y,s}\chi_{+,1}|^2\leq C\exp(-\lambda|y_1|)\quad &\text{if } y_1>1.
\end{aligned}\end{equation}
 \end{thm}
\begin{proof}We only consider the case $y_1\geq 1$, since the decay estimate for $y_1\leq -1$ is totally similar to the proof of  $y_1\geq 1$. Recall that $$w=\chi_j-\psi_+(\chi_{+,j}+\vartheta_j\chi_{+,1})-\psi_-\chi_{-,j}-\partial_1\psi_+\phi_{+,(d+1)1\ell}\partial_\ell P_j-\partial_1\psi_-\phi_{-,(d+1)1\ell}\partial_\ell P_j.$$
First, due to $\partial_sA_{\pm}\in L^\infty(\mathbb{R}^{d+1})$ and $(2.18)$ in Remark 2.3, we have
\begin{equation}
||\partial_s w||_{L^2(\mathbb{D})}\leq C||\partial_s \tilde{f}||_{L^2(\mathbb{D})}+C|| \tilde{f}||_{L^2(\mathbb{D})}\leq C.
\end{equation}
Now, noting $(2.15)$ and  the choices of the cut-off functions $\psi_{\pm}$ in $(2.10)$, we have
\begin{equation}
\partial_s w- \operatorname{div}(A\nabla w)=0 \quad \text{if}\quad y_1>1.
\end{equation}
For any $R_2>R_1> 1$, multiplying the equation above by $\partial_s w$ and integrating the resulting equation over $(y_1,y',s)\in (R_1,R_2)\times \mathbb{T}^d$ yields that
\begin{equation}\begin{aligned}
\int_{(R_1,R_2)\times \mathbb{T}^d} |\partial_s w|^2dyds
=&\frac{1}{2}\int_{(R_1,R_2)\times \mathbb{T}^d}\partial_sA \nabla w \nabla w dyds+\int_{\{R_2\}\times \mathbb{T}^d}A_{1i}\partial_i w\partial_s wdy'ds\\
&-\int_{\{R_1\}\times \mathbb{T}^d}A_{1i}\partial_i w\partial_s wdy'ds\\
\leq& \int_{\{R_2\}\times \mathbb{T}^d}A_{1i}\partial_i w\partial_s wdy'ds-\int_{\{R_1\}\times \mathbb{T}^d}A_{1i}\partial_i w\partial_s wdy'ds\\
&+C_0 ||\nabla w||_{L^2((R_1,R_2)\times \mathbb{T}^d)}^2,
\end{aligned}\end{equation}
where we have used $A\nabla w \nabla\partial_sw=A\nabla\partial_s w \nabla w=1/2 A_{ij}\partial_s(\partial_iw\partial_jw)$ due to $A=A^*$.

Next, for any $R_2>R_1>1$, multiplying the Equation $(2.24)$ by 1 and integrating the resulting equation over $(y_1,y',s)\in (R_1,R_2)\times \mathbb{T}^d$ yields that
\begin{equation*}
\int_{\{R_2\}\times \mathbb{T}^d}A_{1i}\partial_i w dy'ds=\int_{\{R_1\}\times \mathbb{T}^d}A_{1i}\partial_i w dy'ds,
\end{equation*}
which implies that for any $R> 1$, the integral $\int_{\{R\}\times \mathbb{T}^d}A_{1i}\partial_i w dy'ds$ equals to some constant and thus vanishes due to $\nabla w\in L^2((-\infty,\infty)\times \mathbb{T}^d)$. For simplicity, denote $C(R)=:\int_{\{R\}\times \mathbb{T}^d} w dy'ds$. Then multiplying the Equation $(2.24)$ by $w$ and integrating the resulting equation over $(y_1,y',s)\in (R_1,R_2)\times \mathbb{T}^d$ yields that
\begin{equation}\begin{aligned}
&\int_{(R_1,R_2)\times \mathbb{T}^d}A\nabla w\nabla wdyds\\
=&\int_{\{R_2\}\times \mathbb{T}^d}w A_{1i}\partial_i wdy'ds-\int_{\{R_1\}\times \mathbb{T}^d}wA_{1i}\partial_i wdy'ds\\
=&\int_{\{R_2\}\times \mathbb{T}^d}(w-C(R_2)) A_{1i}\partial_i wdy'ds-\int_{\{R_1\}\times \mathbb{T}^d}(w-C(R_1))A_{1i}\partial_i wdy'ds\\
\leq & C||\nabla_{y,s} w||^2_{L^2(\{R_2\}\times \mathbb{T}^d)}+C||\nabla_{y,s} w||^2_{L^2(\{R_1\}\times \mathbb{T}^d)},
\end{aligned}\end{equation}
where we have used the Poincar\'{e}'s inequality in the last inequality.

Therefore, combining $(2.25)$ and $(2.26)$ yields that
\begin{equation}
||\nabla_{y,s} w||^2_{L^2((R_1,R_2)\times \mathbb{T}^d)}\leq C||\nabla_{y,s} w||^2_{L^2(\{R_2\}\times \mathbb{T}^d)}+C||\nabla_{y,s} w||^2_{L^2(\{R_1\}\times \mathbb{T}^d)}.
\end{equation}
Let $R_2\rightarrow \infty$ in $(2.27)$ after in view of $(2.16)$ and $(2.23)$, then Grownwall's Lemma implies that for some positive constant $\lambda$ depending only on $d$, $\kappa$ and $M$,  we have
$$\int_{(R_1,\infty)\times \mathbb{T}^d}|\nabla_{y,s} w|^2dyds\leq C\exp (-\lambda R_1).$$
Thus, we complete the proof for $y_1>1$.
\end{proof}

\begin{rmk}In this remark, we are ready to prove the claim $(2.19)$. For any $R\geq 100$, and according to Remark 2.3,
$$\int_{(1,R)\times \mathbb{T}^d} |\nabla_{y,s} w_R|^2dyds\leq C,$$
with the constant $C$ independent of $R$. Note that $w_R(\pm R,y',s)=0$, for any $(y',s)\in \mathbb{T}^d$, then the estimate above can be improved as
$$\int_{(1,\infty)\times \mathbb{T}^d} |\nabla_{y,s} w_R|^2dyds\leq C,$$
after extending $w_R$ to be 0 outside of $\mathbb{D}_R$.

Next, similar to the proof of Theorem 2.4, for any $1\leq R_1\leq R$, we know that the integral $\int_{\{R_1\}\times \mathbb{T}^d}A_{1i}\partial_i w_R dy'ds$ equals to some constant. Without loss of generality, we assume $\int_{\{R_1\}\times \mathbb{T}^d}A_{1i}\partial_i w_Rdy'ds=c_0\geq 0$. Then by Holder's inequality, there holds
\begin{equation*}\begin{aligned}
c_0(R-1)=&\int_{1}^R\int_{\mathbb{T}^d}A_{1i}\partial_i w_Rdy_1dy'ds\leq C\int_{1}^R\left(\int_{\mathbb{T}^d} |\nabla w_R|^2dy'ds\right)^{1/2}dy_1\\
\leq& C(R-1)^{1/2}\left(\int_{1}^R\int_{\mathbb{T}^d}|\nabla w_R|^2dyds\right)^{1/2}\leq C(R-1)^{1/2},
\end{aligned}\end{equation*}
with the constant $C$ independent of $R$. Therefore, we choose $R\geq C_1$ for some constant $C_1$ suitably large, which immediately implies that $c_0=0$.

To proceed, following the ideas of $(2.25)$-$(2.27)$, we can obtain
$$\int_{(R_1,\infty)\times \mathbb{T}^d}|\nabla_{y,s} w_R|^2dyds\leq C\exp (-\lambda R_1),$$
for any $R_1>1$.

Consequently, for any $R_1\geq 1$ and for any $R\geq C_1$, and according to the local $L^\infty$ estimates (or boundary estimates after noting that we have extended $w_R$ to be 0 outside of $\mathbb{D}_R$) of parabolic theory, there holds
\begin{equation}\begin{aligned}
|w_R(R_1,x_0',t_0)|\leq &C\int_{(R_1-1,R_1+1)\times \mathbb{T}^d}|w_R(R_1,x',t)|\\
=& C\int_{(R_1-1,R_1+1)\times \mathbb{T}^d}|w_R(R_1,x',t)-w_R(\infty,x',t)|\\
\leq &C\int_{R_1-1}^{\infty}\int_{\mathbb{T}^d}|\nabla_{y_1} w_R|\\
\leq &\sum_{n=\lfloor R_1-1\rfloor}^{\infty}\int_{n}^{n+1}\int_{\mathbb{T}^d}|\nabla_{y,s} w_R|\\
\leq &C\sum_{n=\lfloor R_1-1\rfloor }^{\infty}\exp (-\lambda R_1/2)\\
\leq &C\exp (-\lambda R_1/2).
\end{aligned}\end{equation}

As a direct consequence of the inequality above, we have $|w(R_1,x_0',t_0)|\leq C\exp (-\lambda R_1/2)$ after letting $R\rightarrow\infty$, and it is totally similar to the proof above if we consider $w(-R_1,x',t)$ for $R_1>1$. Thus, we have complete the proof of claim $(2.19)$ and
\begin{equation}
|w(y_1,y',s)|\leq C\exp (-\lambda |y_1|/2)\quad \text{if}\quad |y_1|\geq 1.
\end{equation}
\end{rmk}

To proceed, we need the solvability of the following Poisson equation defined in $\mathbb{D}$. Note that in the following several lemmas, $\Delta_{d+1}$, $\nabla_{d+1}$ and $\nabla^2_{d+1}$ denote the Laplacian, gradient and Hessian in $\mathbb{R}^{d+1}$ with $y_{d+1}=s$, respectively.
\begin{lemma}
Let $f\in L^1(\mathbb{D})\cap L^2(\mathbb{D})$, then there exists a solution $u$ to the Poisson equation $\Delta_{d+1} u=f$ in $\mathbb{D}$, such that $\nabla_{d+1} u\in L^\infty(\mathbb{D})+ L^2(\mathbb{D})$ and $\nabla^2_{d+1} u\in L^2(\mathbb{D})$.
\end{lemma}
\begin{proof}Actually, this result has been proved in \cite{zhang2023quantitative,MR3974127}, and we provide it for completeness.
First, we decompose $f(y)=f_1(y_1)+f_2(y)$, with
$$f_1(y_1)=\fint_{\mathbb{T}^{d}}f(y_1,y')dy',$$
and $y_{d+1}=s$. Then it is easy to see that
\begin{equation}
f_1,f_2\in L^1(\mathbb{D})\cap L^2(\mathbb{D}),\ \fint_{\mathbb{T}^{d}}f_2(y_1,y')dy'=0.
\end{equation}
Denote $$N_1(y_1)=\int_{0}^{y_1}\int_{0}^{t}f_1(s)dsdt,$$
then $\Delta_{d+1} N_1(y_1)=\partial_1^2 N_1(y_1)=f_1(y_1)$. Consequently, according to $(2.30)$, $\nabla_{d+1} N_1\in L^\infty(\mathbb{D})$ and $\nabla^2_{d+1} N_1\in L^2(\mathbb{D})$.

To continue, by viewing $y_1$ as a parameter after noting $(2.30)$, we can solve the following equation,
\begin{equation}\Delta_{y'}N_2(y_1,y')=f_2(y_1,y')\text{ in }\mathbb{T}^{d},\ \text{ with }\fint_{\mathbb{T}^{d}}N_2(y_1,y')dy'=0.\end{equation}
Define the $\mathbb{D}$-periodic function $g=:(0,\partial_2 N_2,\cdots, \partial_{d+1} N_2)$.
 By energy estimates, we have $||\nabla_{y'}N_2(y_1,\cdot)||_{L^2(\mathbb{T}^{d})}\leq C ||f_2(y_1,\cdot)||_{L^2(\mathbb{T}^{d})}$, with $C$ being independent of $y_1$, which further implies that $||\nabla_{y'}N_2||_{L^2(\mathbb{D})}\leq C ||f_2||_{L^2(\mathbb{D})}$ and $||g||_{L^2(\mathbb{D})}\leq C ||f_2||_{L^2(\mathbb{D})}$.

According to the Equation $(2.31)$, we can express $f_2$ as $f_2=\operatorname{div}_yg$ with $g\in L^2({\mathbb{D}})$, then by the
classical Lax-Milgram Theorem, there exists a unique (up to the addition of constant) $\tilde{N}_2$ solving the
Poisson equation $\Delta_{d+1} \tilde{N}_2=\operatorname{div}g$ with $\nabla_{d+1} \tilde{N}_2\in  L^2({\mathbb{D}})$.
Moreover, due to $\Delta_{d+1} \tilde{N}_2=f_2$, $\nabla_{d+1} \tilde{N}_2\in  L^2({\mathbb{D}})$ and $f_2\in L^1(\mathbb{D})\cap L^2(\mathbb{D})$, we have $\nabla^2_{d+1} \tilde{N}_2\in L^2(\mathbb{D})$ according to the regularity theory.

Consequently, $N_1+\tilde{N}_2$ is the desired solution satisfying $\nabla_{d+1}  (N_1+\tilde{N}_2)\in L^\infty(\mathbb{D})+ L^2(\mathbb{D})$ and $\nabla^2_{d+1} (N_1+\tilde{N}_2)\in L^2(\mathbb{D})$.\end{proof}

To introduce the following result, we denote
$$\mathcal{H}=\left\{u:\nabla_{d+1} u\in L^2(\mathbb{D}), \fint_{\mathbb{D}_1}u=0\right\},$$
which is a Hilbert space equipped with the inner product: $$\langle u,v\rangle_{\mathcal{H}}=:\int_{\mathbb{D}}\nabla_{d+1} u \cdot \nabla_{d+1} vdy.$$

\begin{lemma}
Let $f\in L^2(\mathbb{D})$ with $supp(f)\subset \mathbb{D}_1$, then there exists a unique solution $u\in \mathcal{H}$ to the Poisson equation $\Delta_{d+1} u=f$ in $\mathbb{D}$, with the energy estimate $||\nabla_{d+1} u||_{L^2(\mathbb{D})}\leq C ||f||_{L^2(\mathbb{D})}$, for the constant $C$ depending only on $d$. Moreover,
due to the regularity theory, we have $||\nabla^2_{d+1} u||_{L^2(\mathbb{D})}\leq C ||f||_{L^2(\mathbb{D})}$.
\end{lemma}
\begin{proof}
For any $v\in \mathcal{H}$, we define
\begin{equation*}\begin{aligned}
F(v)=:\int_{\mathbb{D}}fvdy=\int_{\mathbb{D}_1}fvdy=\int_{\mathbb{D}_1}f\left(v-\fint_{\mathbb{D}_1}v\right)dy,
\end{aligned}\end{equation*}

which, by Poincar\'{e}'s inequality, implies that
\begin{equation*}\begin{aligned}
|F(v)|\leq C||f||_{L^2(\mathbb{D})}||\nabla_{d+1} v||_{L^2(\mathbb{D})}.
\end{aligned}\end{equation*}

Thus, $F$ is a bounded linear functional on $\mathcal{H}$. Then according to the calssical Lax-Milgram Theorem, there exists a unique solution $u\in \mathcal{H}$ to  the Poisson equation $\Delta_{d+1} u=f$ in $\mathbb{D}$, satisfying the energy estimate $||\nabla_{d+1} u||_{L^2(\mathbb{D})}\leq C ||f||_{L^2(\mathbb{D})}$.

Consequently, we complete this proof.
%
%
\end{proof}

\begin{thm}[dual correctors] Let $1\leq j\leq d$ and assume that the matrix $A(y,s)$ satisfies the conditions $(1.2)$-$(1.4)$. Then there exist $\mathbb{D}$-periodic functions $\phi_{KIj}(y,s)$ in $\mathbb{R}^{d+1}$ such that \begin{equation}\begin{aligned}
&\phi_{KIj}-\psi_+\phi_{+,KI\ell}\partial_\ell P_j-\psi_-\phi_{-,KIj}\in L^\infty(\mathbb{D})+L^2(\mathbb{D}),\\ &\nabla_{d+1}\phi_{KIj}-\psi_+\nabla_{d+1}\phi_{+,KI\ell}\partial_\ell P_j-\psi_-\nabla_{d+1}\phi_{-,KIj}\in L^2(\mathbb{D}),
\end{aligned}\end{equation}
and
\begin{equation}
B_{Ij}=\partial_{y_K}\phi_{KIj}\quad\text{and}\quad \phi_{KIj}=-\phi_{IKj}\quad \text{in }\mathbb{D},
\end{equation}
where $1\leq K, I\leq d+1$, $B_{ij}=:\widehat{A}_{i\ell}\partial_\ell P_j-A_{i\ell}(\partial_\ell P_j+\partial_\ell \chi_j)$ and $B_{(d+1)j}=-\chi_{j}$  for $1\leq i,j,\ell\leq d$, and $\phi_{\pm}$ are given by Lemma 2.1. Note that we use the notation $y_{d+1}=s$.\end{thm}
\begin{proof}
Similar to the periodic case without an interface, we need to solve the equation
\begin{equation}\Delta_{d+1}f_{{I}j}=B_{{I}j}\quad \text{in }\mathbb{D},\quad \text{with }y_{d+1}=s,\   1\leq{I}\leq d+1\text{ and }1\leq j\leq d.\end{equation}

Following the method in \cite{MR3421758,MR3974127}, we decompose $f_{{I}j}$ as
\begin{equation*}
f_{{I}j}=\psi_+f_{+,{I}\ell}\partial_\ell P_j+\psi_-f_{-,{I}\ell}\partial_\ell P_j+\tilde{f}_{{I}j}.
\end{equation*}

Recall that $\nabla P$ is piecewise constant and possibly discontinuous only across the interface, where $\psi_{\pm}$ vanishes. Hence, a direct computation shows that

$$\begin{aligned}\Delta_{d+1}\tilde{f}_{ij}=&\Delta_{d+1}{f}_{ij}-\psi_+\Delta_{d+1}f_{+,{i}\ell}\partial_\ell P_j-\psi_-\Delta_{d+1}f_{-,{i}\ell}\partial_\ell P_j\\
&-2\left(\partial_{1}\psi_+\partial_{1}f_{+,i\ell}\partial_\ell P_j+\partial_{1}\psi_-\partial_{1}f_{-,i\ell}\partial_\ell P_j\right)\\
&-\partial_{11}\psi_+f_{+,i\ell}\partial_\ell P_j-\partial_{11}\psi_-f_{-,i\ell}\partial_\ell P_j.\end{aligned}$$
Using $(2.8)$ and $(2.34)$ yields that
$$\begin{aligned}
&\Delta_{d+1}{f}_{ij}-\psi_+\Delta_{d+1}f_{+,{i}\ell}\partial_\ell P_j-\psi_-\Delta_{d+1}f_{-,{i}\ell}\partial_\ell P_j\\
=&(1-\psi_--\psi_+)\left(\widehat{A}_{i\ell}\partial_\ell P_j-A_{i\ell}(\partial_\ell P_j+\partial_\ell \chi_j)\right)\\
&+\psi_+A_{ik}\left(\partial_k\chi_{+,\ell}\partial_\ell P_j-\partial_k\chi_{j}\right)+
\psi_-A_{ik}\left(\partial_k\chi_{-,\ell}\partial_\ell P_j-\partial_k\chi_{j}\right),
\end{aligned}$$
which immediately implies that
\begin{equation}\begin{aligned}
\Delta_{d+1}\tilde{f}_{ij}=&(1-\psi_--\psi_+)\left(\widehat{A}_{i\ell}\partial_\ell P_j-A_{i\ell}(\partial_\ell P_j+\partial_\ell \chi_j)\right)
+\psi_+A_{ik}\left(\partial_k\chi_{+,\ell}\partial_\ell P_j-\partial_k\chi_{j}\right)\\
&+\psi_-A_{ik}\left(\partial_k\chi_{-,\ell}\partial_\ell P_j-\partial_k\chi_{j}\right)-2\left(\partial_{1}\psi_+\partial_{1}f_{+,i\ell}\partial_\ell P_j+\partial_{1}\psi_-\partial_{1}f_{-,i\ell}\partial_\ell P_j\right)\\
&-\partial_{11}\psi_+f_{+,i\ell}\partial_\ell P_j-\partial_{11}\psi_-f_{-,i\ell}\partial_\ell P_j.
\end{aligned}\end{equation}
Similarly, there holds
\begin{equation}\begin{aligned}
\Delta_{d+1}\tilde{f}_{(d+1)j}=&-(1-\psi_--\psi_+)\chi_j +\psi_+\left(\chi_{j}-\chi_{+,\ell} \partial_\ell P_j\right)\\
&+\psi_-\left(\chi_{j}-\chi_{-,\ell} \partial_\ell P_j\right)
-\partial_{11}\psi_+f_{+,(d+1)\ell}\partial_\ell P_j-\partial_{11}\psi_-f_{-,(d+1)\ell}\partial_\ell P_j\\
&-2\left(\partial_{1}\psi_+\partial_{1}f_{+,(d+1)\ell}\partial_\ell P_j+\partial_{1}\psi_-\partial_{1}f_{-,(d+1)\ell}\partial_\ell P_j\right).
\end{aligned}\end{equation}
To continue, we need to find a solution $\tilde{f}_{ij}$ to Equation $(2.35)$. Set
$$\begin{aligned}
&M_{1,ij}=(1-\psi_--\psi_+)\left(\widehat{A}_{i\ell}\partial_\ell P_j-A_{i\ell}(\partial_\ell P_j+\partial_\ell \chi_j)\right),\\
M_{2,ij}&=\psi_+A_{ik}\left(\partial_k\chi_{+,\ell}\partial_\ell P_j-\partial_k\chi_{j}\right)+\psi_-A_{ik}\left(\partial_k\chi_{-,\ell}\partial_\ell P_j-\partial_k\chi_{j}\right),\\
M_{3,ij}=-\partial_{11}&\psi_+f_{+,i\ell}\partial_\ell P_j-\partial_{11}\psi_-f_{-,i\ell}\partial_\ell P_j-2\left(\partial_{1}\psi_+\partial_{1}f_{+,i\ell}\partial_\ell P_j+\partial_{1}\psi_-\partial_{1}f_{-,i\ell}\partial_\ell P_j\right).
\end{aligned}$$

Firstly, it follows from Lemma 2.7 that there exists a $\mathbb{D}$-periodic solution $\tilde{f}_{1,ij}$ to the equation $\Delta_{d+1}\tilde{f}_{1,ij}=M_{1,ij}$ in $\mathbb{D}$, satisfying $\nabla_{d+1} \tilde{f}_{1,ij}\in H^1(\mathbb{D})$.

Secondly, following the idea of $(2.28)$, we have $M_{2,ij}\in L^1(\mathbb{D})\cap L^2(\mathbb{D})$ due to $(2.9)$ and $(2.22)$. Then according to
Lemma 2.6, there exists a $\mathbb{D}$-periodic solution $\tilde{f}_{2,ij}$ to the equation $\Delta_{d+1}\tilde{f}_{2,ij}=M_{2,ij}$ in $\mathbb{D}$,
satisfying $\nabla_{d+1} \tilde{f}_{2,ij}\in L^\infty(\mathbb{D})+ L^2(\mathbb{D})$ and $\nabla^2_{d+1} \tilde{f}_{2,ij}\in L^2(\mathbb{D})$.

Thirdly, it follows from Lemma 2.7 again that there exists a $\mathbb{D}$-periodic solution $\tilde{f}_{3,ij}$ to the equation $\Delta_{d+1}\tilde{f}_{3,ij}=M_{3,ij}$ in $\mathbb{D}$, satisfying $\nabla_{d+1} \tilde{f}_{3,ij}\in H^1(\mathbb{D})$.

In conclusion, we have constructed a $\mathbb{D}$-periodic solution $\tilde{f}_{ij}$ to the Equation $(2.35)$, satisfying $\nabla_{d+1}\tilde{f}_{ij}\in L^\infty(\mathbb{D})+ L^2(\mathbb{D})$ and $\nabla^2_{d+1} \tilde{f}_{ij}\in L^2(\mathbb{D})$.

Totally similar to the discussion above after noting $(2.29)$, we can construct a solution $\tilde{f}_{(d+1)j}$ to
the Equation $(2.36)$, satisfying $\nabla_{d+1}\tilde{f}_{(d+1)j}\in L^\infty(\mathbb{D})+ L^2(\mathbb{D})$ and $\nabla^2_{d+1} \tilde{f}_{(d+1)j}\in L^2(\mathbb{D})$. Actually, the functions
$\tilde{f}_{(d+1)j}$ have better regularity which we do not need.

Therefore, back to $(2.34)$, we have constructed a solution $f_{Ij}$ to the Equation $(2.34)$, satisfying
\begin{equation}\begin{aligned}
&\nabla_{d+1}f_{Ij}-\psi_+\nabla_{d+1}f_{+,{I}\ell}\partial_\ell P_j-\psi_-\nabla_{d+1}f_{-,{I}\ell}\partial_\ell P_j\in L^\infty(\mathbb{D})+L^2(\mathbb{D}),\\ &\nabla^2_{d+1}f_{Ij}-\psi_+\nabla^2_{d+1}f_{+,{I}\ell}\partial_\ell P_j-\psi_-\nabla^2_{d+1}f_{-,{I}\ell}\partial_\ell P_j\in L^2(\mathbb{D}).
\end{aligned}\end{equation}
 To proceed, for $1\leq K,I\leq d+1$, set
\begin{equation}B_{Ij}=\partial_K\left(\partial_K f_{Ij}-\partial_If_{Kj}\right)+\partial_I\left(\partial_Kf_{Kj}\right).\end{equation}
Note that by $(2.5)$,
$$\sum_{I=1}^{d+1}\partial_IB_{Ij}=\sum_{i=1}^{d}\partial_iB_{ij}-\partial_s \chi_j=0\quad \text{in }\mathbb{D}.$$
It follows from $(2.34)$ that
$$\sum_{I=1}^{d+1}\partial_I f_{Ij}$$
is a harmonic function in $\mathbb{R}^{d+1}$. To see that it is a constant, we first note that
$$\sum_{I=1}^{d+1}\partial_I f_{+,Ij},\ \sum_{I=1}^{d+1}\partial_I f_{-,Ij}$$
are constants (see Lemma 2.1). Then according to $(2.37)_2$, we know that
$$\begin{aligned}
&\|\nabla_{d+1} \sum_{I=1}^{d+1}\partial_I f_{Ij}\|_{L^2(\mathbb{D_+})}=\|\nabla_{d+1} \sum_{I=1}^{d+1}\partial_I f_{Ij}-\nabla_{d+1} \sum_{I=1}^{d+1}\partial_I f_{+,I\ell}\partial_\ell P_j\|_{L^2(\mathbb{D_+})}<\infty,\\
&\|\nabla_{d+1} \sum_{I=1}^{d+1}\partial_I f_{Ij}\|_{L^2(\mathbb{D_-})}=\|\nabla_{d+1} \sum_{I=1}^{d+1}\partial_I f_{Ij}-\nabla_{d+1} \sum_{I=1}^{d+1}\partial_I f_{-,I\ell}\partial_\ell P_j\|_{L^2(\mathbb{D_-})}<\infty,
\end{aligned}$$
which implies that $\nabla_{d+1} \sum_{I=1}^{d+1}\partial_I f_{Ij}\in L^2(\mathbb{D})$. Hence $\sum_{I=1}^{d+1}\partial_I f_{Ij}$ is a constant since it is a harmonic function in $\mathbb{D}$.

 Consequently, by $(2.38)$, we obtain
$$B_{Ij}=\partial_{K}\phi_{KIj}\quad \text{in }\mathbb{D},$$ where
$$\phi_{KIj}=\partial_{K}f_{Ij}-\partial_{I}f_{Kj}$$ is $\mathbb{D}$-periodic. It is easy to see $\phi_{KIj}=-\phi_{IKj}$. Moreover, the desired estimates $(2.32)$ follow from $(2.1)$, $(2.37)$ and the facts $\phi_{\pm,KIj}=\partial_{K}f_{\pm,Ij}-\partial_{I}f_{\pm,Kj}$.
\end{proof}

The next result recovers the regularity of $u_0$ after suitable transformation. Precisely, as the same
observation in
\cite{MR3974127},
 if we set $\tilde{u}_0(z,t)=:u_0(P^{-1}(z),t)$ and
$\tilde{U}_0(x,t)=:(\nabla P(x))^{-1}\cdot \nabla u_0(x,t)$ for $0<t\leq T$, then it is easy to see that
$\partial_{z_j}\tilde{u}_0=\tilde{U}_{0,j}(P^{-1}(z),t)$. Therefore, due to $\widehat{A}\nabla
u_0=\widehat{A}\nabla P\cdot (\nabla P)^{-1}\nabla u_0$ and the transmission conditions through the interface (see $(2.4)$), the function $\tilde{U}_{0,j}$ (as well as $\partial_{z_j}\tilde{u}_0$) is continuous across the
interface $\mathcal{I}$ at least for sufficient regular source term $f$.

For $\tilde{u}_0(z,t)$, we have the following $W^{2,1}_2$-estimates.

\begin{thm}[$W^{2,1}_2$ estimates]Let $u_0$ be the solution to the homogenized equation $(1.7)$ with $f\in L^2(\Omega_T)$, and denote $\tilde{u}_0(z,t)=:u_0(P^{-1}(z),t)$, then
\begin{equation}
||\tilde{u}_0||_{W^{2,1}_2(\Omega_T)}\leq C ||f||_{L^2(\Omega_T)},
\end{equation}
where the constant $C$ depends only on $d$, $\kappa$ and $T$. Moreover, if we additionally assume $f_t\in L^2(\Omega_T)$, then
\begin{equation}
\|\tilde{\eta} \nabla\partial_t{u}_0\|_{L^2(\Omega_T)}\leq C \varepsilon^{-1}||f||_{L^2(\Omega_T)\cap L^\infty(0,T;L^2(\mathbb{R}^d))}+C||f_t||_{L^2(\Omega_T)},
\end{equation}
where the constant $C$ depends only on $d$, $\kappa$ and $T$, and $\tilde{\eta}(t)\in [0,1]$ is a cut-off function  satisfying $\tilde{\eta}=0$ if $t\leq 2 \varepsilon^2$ and $\tilde{\eta}=1$ if $t\geq  4 \varepsilon^2$ with $|\tilde{\eta}'|\leq C\varepsilon^{-2}$.
\end{thm}

\begin{proof}
Note that the function $\tilde{u}_0(z,t)=:u_0(P^{-1}(z),t)$
satisfies the following parabolic equation:
\begin{equation}|J(z)|^{-1}\partial_t\tilde{u}_0 -\operatorname{div}\left(|J(z)|^{-1} \tilde{A}(z,t) \cdot \nabla \tilde{u}\right)=|J(z)|^{-1} f\left(P^{-1}(z)\right)\ \text{in}\  \mathbb{R}^d\times (0,T),\end{equation}
with $\tilde{u}_0(z,0)=0$ for $z\in\mathbb{R}^d$,
where $\tilde{A}(z,t)$ is defined by
$$\tilde{A}(z):=\left(\nabla P\left(P^{-1}(z)\right)\right)^{T} \cdot \widehat{A}\left(P^{-1}(z)\right) \cdot \nabla P\left(P^{-1}(z)\right),$$
and $J(z)$ is the Jacobian of $P$ evaluated on $P^{-1}(z)$. By construction, $\tilde{A}(z)$ is elliptic and constant on the half-spaces $\mathbb{R}_\pm^*\times \mathbb{R}^{d-1}$, and the product $|J(x)|^{-1} \tilde{A}(x)$ is divergence-free in $\mathbb{R}^d$ due to $(2.3)$. Therefore, we can rewrite the Equation $(2.41)$ as
\begin{equation}\partial_t\tilde{u}_0(z,t)-\tilde{A}_{ij}(z)\partial_{ij}\tilde{u}(z,t)=f(P^{-1}(z),t)\quad \text{in}\quad  \mathbb{R}^d\times (0,T).\end{equation}

\noindent
Thus, applying \cite[Theorem 2.2]{MR2300337} yields that
\begin{equation*}||\tilde{u}_0||_{W^{2,1}_2(\mathbb{R}^d\times (0,T))}\leq C||f||_{L^2(\mathbb{R}^d\times (0,T))},\end{equation*}
which completes the proof of $(2.39)$.
\noindent
Moreover, $\tilde{\eta}\partial_t{u}_0$ satisfies
\begin{equation}
\partial_t (\tilde{\eta} \partial_t{u}_0) -\operatorname{div}(\widehat{A}\nabla (\tilde{\eta} \partial_t{u}_0))=\partial_t(\tilde{\eta} f)+\operatorname{div}(\widehat{A}\nabla {u}_0\partial_t\tilde{\eta})\quad \text{in}\quad\mathbb{R}^d\times (0,T),\end{equation}
with $(\tilde{\eta} \partial_t{u}_0)(z,0)=0$ for $z\in\mathbb{R}^d$.\\

Then, energy estimates yield that
\begin{equation}\begin{aligned}
&\sup_{0\leq t\leq T}\int_{\mathbb{R}^d}|\tilde{\eta} \partial_t{u}_0|^2+\iint_{\Omega_T}|\tilde{\eta} \nabla \partial_t u_0|^2\\
\leq &C\varepsilon^{-2}\int_{0}^{4\varepsilon^2}\int_{\mathbb{R}^d}|f|\cdot |\tilde{\eta} \partial_t{u}_0|+C\iint_{\Omega_T}|f_t|\cdot |\tilde{\eta} \partial_t{u}_0|+C\varepsilon^{-2}\int_{0}^{4\varepsilon^2}\int_{\mathbb{R}^d}|\nabla u_0|\cdot |\tilde{\eta} \nabla \partial_t u_0|\\
\leq &C\varepsilon^{-1}||f||_{L^2(\Omega_T)}\sup_{0\leq t\leq T}||\tilde{\eta} \partial_t{u}_0(\cdot,t)||_{L^2(\mathbb{R}^d)}+CT^{1/2}||f_t||_{L^2(\Omega_T)}\sup_{0\leq t\leq T}||\tilde{\eta} \partial_t{u}_0(\cdot,t)||_{L^2(\mathbb{R}^d)}\\
&\quad\quad+C\varepsilon^{-2}||\tilde{\eta} \nabla\partial_t{u}_0||_{L^2(\Omega_{4\varepsilon^2})}||\nabla u_0||_{L^2(\Omega_{4\varepsilon^2})}.
\end{aligned}\end{equation}
\noindent
To see the estimate on $||\nabla u_0||^2_{L^2(\Omega_{4\varepsilon^2})}$, we note that
\begin{equation*}
\int_{\mathbb{R}^d} \widehat{A} \nabla u_{0} \cdot \nabla u_{0}=-\int_{\mathbb{R}^d} \partial_{t} u_{0} \cdot u_{0}+\int_{\mathbb{R}^d} f \cdot u_{0},
\end{equation*}
then \begin{equation}
\begin{aligned}
\kappa \int_{0}^{4\varepsilon^2} \int_{\mathbb{R}^d}\left|\nabla u_{0}\right|^{2} & \leq \int_{0}^{4\varepsilon^2} \int_{\mathbb{R}^d}\left|\partial_{t} u_{0} \| u_{0}\right|+\int_{0}^{4\varepsilon^2} \int_{\mathbb{R}^d}|f|\left|u_{0}\right| \\
& \leq\left\{\left\|\partial_{t} u_{0}\right\|_{L^{2}\left(\Omega_{4\varepsilon^2}\right)}+\|f\|_{L^{2}\left(\Omega_{4\varepsilon^2}\right)}\right\}
\left(\int_{0}^{4\varepsilon^2} \int_{\mathbb{R}^d}\left|u_{0}\right|^{2}\right)^{1 / 2} \\
& \leq C\varepsilon\|f\|_{L^{2}\left(\Omega_{4\varepsilon^2}\right)}\sup_{0\leq t\leq T}\left\|u_{0}(\cdot, t)\right\|_{L^{2}(\mathbb{R}^d)}\\
&\leq C\varepsilon^2||f||_{L^{\infty}(0,T;L^2(\mathbb{R}^d))}||f||_{L^{2}(\Omega_T)},
\end{aligned}
\end{equation}
where we have used the $W^{2,1}_2$-estimates $\|\partial_{t} \tilde{u}_{0}\|_{L^{2}\left(\Omega_{4\varepsilon^2}\right)}\leq C \|f\|_{L^{2}\left(\Omega_{4\varepsilon^2}\right)}$ with $\tilde{u}_{0}(z,0)=0$ for $z\in \mathbb{R}^d$.

Consequently, combining $(2.44)$-$(2.45)$ yields the desired estimate
\begin{equation*}
\|\tilde{\eta} \nabla\partial_t{u}_0\|_{L^2(\Omega_T)}\leq C \varepsilon^{-1}||f||_{L^2(\Omega_T)\cap L^\infty(0,T;L^2(\mathbb{R}^d))}+C||f_t||_{L^2(\Omega_T)}.
\end{equation*}
\end{proof}

\section{Convergence rates}
\subsection{Smoothing Operator}
 Fix two nonnegative function $\rho_1\in C_0^\infty(B(0,1/2))$ and $\rho_2\in C_0^\infty((-1/2,1/2))$, where $B(0,1/2)$ is the ball in $\mathbb{R}^d$ of radius 1/2 centered at the origin 0,  such that
$$\int_{\mathbb{R}^d}\rho_1(y) dy=1\quad\text{and}\quad\int_{\mathbb{R}}\rho_2(s) ds=1.$$
For $\varepsilon>0$, let $\rho_{1,\varepsilon}(y)=\varepsilon^{-d}\rho(y/\varepsilon)$ and $\rho_{2,\varepsilon}(s)=\varepsilon^{-2}\rho(s/\varepsilon)$. For a multi-variable function $f(x,t)$ on $\mathbb{R}^{d+1}$, define
\begin{equation}\begin{aligned}
S^x_\varepsilon(f)(x,t)=\rho_{1,\varepsilon}\ast f(x,t)=\int_{\mathbb{R}^d}f(x-y,t)\rho_{1,\varepsilon}(y)dy,\\
S^t_\varepsilon(f)(x,t)=\rho_{2,\varepsilon}\ast f(x,t)=\int_{\mathbb{R}}f(x,t-s)\rho_{2,\varepsilon}(s)ds.\end{aligned}\end{equation}

Moreover, we use the following notation:

\begin{equation}S_\varepsilon=:S_\varepsilon^t\circ S^x_\varepsilon=S^x_\varepsilon\circ S_\varepsilon^t. \end{equation}

In the following lemmas, we collect and state some basic and useful estimates related to the smoothing operator.

\begin{lemma}
Let $S_\varepsilon$ be defined as in $(3.1)$ and $(3.2)$. For $1\leq p<\infty$, then

\begin{equation}
\begin{aligned}
&\quad\quad\quad\quad \quad  \left\|S_{\varepsilon}(f)\right\|_{L^{2}\left(\mathbb{R}^{d+1}\right)} \leq\|f\|_{L^{2}\left(\mathbb{R}^{d+1}\right)},\\
&\quad \quad\quad\quad \varepsilon^{2}\left\|\partial_{t} S_{\varepsilon}(f)\right\|_{L^{2}(\mathbb{R}^{ d+1})} \leq C\|f\|_{L^{2}\left(\mathbb{R}^{d+1}\right)},\\
&\varepsilon\left\|\nabla S_{\varepsilon}(f)\right\|_{L^{2}\left(\mathbb{R}^{d+1}\right)}+\varepsilon^{2}\left\|\nabla^{2} S_{\varepsilon}(f)\right\|_{L^{2}\left(\mathbb{R}^{d+1}\right)} \leq C\|f\|_{L^{2}\left(\mathbb{R}^{d+1}\right)},\\
&\left\| S_{\varepsilon}(\nabla f)-\nabla f\right\|_{L^{p}(\mathbb{R}^{d+1})} \leq C \varepsilon\left\{\left\|\nabla^{2} f\right\|_{L^{p}(\mathbb{R}^{d+1})}+\left\|\partial_{t} f\right\|_{L^{p}\left(\mathbb{R}^{d+1}\right)}\right\},
\end{aligned}
\end{equation}
where $C$ depends only on $d$ and $p$.
\end{lemma}
\begin{proof}The proofs of $(3.3)_1$-$(3.3)_3$ can be found in \cite[Lemma 3.1,Lemma 3.2]{MR3596717} and we need only to prove $(3.3)_4$.
Actually, the estimate $(3.3)_4$ has already been obtained in \cite{xu2020convergence}, and we provide it for completeness.
Firstly, we note that
\begin{equation}\begin{aligned}
\left\| S_{\varepsilon}(\nabla f)-\nabla f\right\|_{L^{p}(\mathbb{R}^{d+1})}\leq \left\| S^x_{\varepsilon}(\nabla f)-\nabla f\right\|_{L^{p}(\mathbb{R}^{d+1})}+\left\| S_{\varepsilon}(\nabla f)-S^x_\varepsilon(\nabla f)\right\|_{L^{p}(\mathbb{R}^{d+1})}.
\end{aligned}\end{equation}
For the first term in the r.h.s. of  $(3.4)$, using \cite[Proposition 3.1.6]{MR3838419} yields that

$$\left\| S^x_{\varepsilon}(\nabla f)-\nabla f\right\|_{L^{p}(\mathbb{R}^{d+1})}\leq C \varepsilon \left\|\nabla^{2} f\right\|_{L^{p}(\mathbb{R}^{d+1})}.$$
For the second term in the r.h.s. of $(3.4)$, using the following Poincar\'{e}'s inequality
\begin{equation*}
\int_{B(x, r)}|u(y)-u(z)| d y \leq C r^{d} \int_{B(x, r)}|\nabla u(y)||y-z|^{1-d} d y,
\end{equation*}
for any $u\in C^1(B(x, r))$ and $z\in B(x, r)$, we have
\begin{equation}
\begin{aligned}
\left|S_{\varepsilon}^{x}(\nabla f)(t)-S_{\varepsilon}(\nabla f)(t)\right| & \leq C \fint_{\left\{|\varsigma-t| \leq \frac{\varepsilon^{2}}{2}\right\}}\left|S_{\varepsilon}^{x}(\nabla f)(t)-S_{\varepsilon}^{x}(\nabla f)(\varsigma)\right| d \varsigma \\
& \leq C \varepsilon^{2} \fint_{\left\{|\varsigma-t| \leq \frac{\varepsilon^{2}}{2}\right\}}\left|\partial_{t} S_{\varepsilon}^{x}(\nabla f)(\varsigma)\right| d \varsigma,
\end{aligned}
\end{equation}
which, together with Young's inequality, yields that
\begin{equation}
\begin{aligned}
\left\|S_{\varepsilon}^{x}(\nabla f)-S_{\varepsilon}(\nabla f)\right\|_{L^{p}(\mathbb{R}^{d+1} )} & \leq C \varepsilon\left\|\fint_{\left\{|\varsigma-t| \leq \frac{\varepsilon^{2}}{2}\right\}}\left|\left(\nabla \rho_{1}\right)_{\varepsilon} *\left(\partial_{t} f\right)(\varsigma)\right| d \varsigma\right\|_{L^{p}(\mathbb{R} ; L^{p}(\mathbb{R}^{d}))} \\
& \leq C \varepsilon\left\|\partial_{t} f\right\|_{L^{p}(\mathbb{R} ; L^{p}(\mathbb{R}^{d}))}.
\end{aligned}
\end{equation}

Consequently, combining $(3.4)$-$(3.6)$ yields the desired estimate $(3.3)_4$.
\end{proof}

\begin{lemma}
Let $f=f(x,t)$ be a $\mathbb{D}$-periodic function. Then
\begin{equation}
||f^\varepsilon S_\varepsilon(g)||_{L^p(\mathbb{R}^{d+1})}\leq C ||f||_{L^p(\mathbb{D})}||g||_{L^p(\mathbb{R}^{d+1})},
\end{equation}for any $1\leq p<\infty$, where $f^\varepsilon(x,t)=:f(x/\varepsilon,t/\varepsilon^2)$ and $C$ depends only on $d$ and $p$.
\end{lemma}

\begin{proof}
By a change of variables, it suffices to consider the case $\varepsilon=1$. In this case, we use $\int_{\mathbb{R}^{d}}\rho_1=1$, $\int_{\mathbb{R}}\rho_2=1$ and Holder's inequality to obtain
\begin{equation}
\left|S_{1}(g)(x, t)\right|^{p} \leq \int_{\mathbb{R}^{d+1}}|g(y, s)|^{p} \rho_1(x-y)\rho_2( t-s) d y d s.
\end{equation}
It follows by Fubini's Theorem that
\begin{equation}\begin{aligned}
\int_{\mathbb{R}^{d+1}}|f(x, t)|^{p}\left|S_{1}(g)(x, t)\right|^{p} d x d t & \leq C \sup _{(y, s) \in \mathbb{R}^{d+1}} \int_{Q((y, s), 1)}|f(x, t)|^{p} d x d t \int_{\mathbb{R}^{d+1}}|g(y, s)|^{p} d y d s \\
& \leq C||f||^p_{L^p(\mathbb{D})}\|g\|_{L^{p}(\mathbb{R}^{d+1})}^{p},
\end{aligned}\end{equation}
where $Q((y,s),1)=:B(y,1)\times(s-1,s+1)$ and $C$ depends only on $d$. This gives for the case $\varepsilon=1$. Moreover, from $(3.9)$, we note that if
$f$ is 1-periodic in $(x,t)$, then
\begin{equation}
||f^\varepsilon S_\varepsilon(g)||_{L^p(\mathbb{R}^{d+1})}\leq C ||f||_{L^p(\mathbb{T}^{d+1})}||g||_{L^p(\mathbb{R}^{d+1})}.
\end{equation}
\end{proof}

As a direct consequence of Lemma 3.2, we have the following result:

\begin{cor}
Let $f_1(x,t)$ be a $\mathbb{D}$-periodic function and $f_2(x,t)$ be a 1-periodic function in $(x,t)$, such that $f_1-f_2\in {L^p(\mathbb{D})}$ for some $1\leq p<\infty$. Then
\begin{equation}
||f_1^\varepsilon S_\varepsilon(g)||_{L^p(\mathbb{R}^{d+1})}\leq C ||f_1-f_2||_{L^p(\mathbb{D})}||g||_{L^p(\mathbb{R}^{d+1})}+C||f_2||_{L^p(\mathbb{T}^{d+1})}||g||_{L^p(\mathbb{R}^{d+1})},
\end{equation}
where $C$ depends only on $d$ and $p$.
\end{cor}
\begin{proof}
According to $(3.9)$-$(3.10)$, it is easy to see that
$$\begin{aligned}
||f_1^\varepsilon S_\varepsilon(g)||_{L^p(\mathbb{R}^{d+1})}&\leq||(f_1 -f_2)^\varepsilon S_\varepsilon(g)||_{L^p(\mathbb{R}^{d+1})}+ ||f_2^\varepsilon S_\varepsilon(g)||_{L^p(\mathbb{R}^{d+1})}\\
&\leq C||f_1-f_2||_{L^p(\mathbb{D})}||g||_{L^p(\mathbb{R}^{d+1})}+
C||f_2||_{L^p(\mathbb{T}^{d+1})}||g||_{L^p(\mathbb{R}^{d+1})},
\end{aligned}$$
thus we complete this proof.
\end{proof}

\begin{rmk}
The same argument as in the proofs of Lemma 3.2 and Corollary 3.3 also shows that
\begin{equation}\begin{aligned}
&\left\|f^{\varepsilon} \nabla S_{\varepsilon}(g)\right\|_{L^{p}\left(\mathbb{R}^{d+1}\right)} \leq C \varepsilon^{-1}\|f\|_{L^{p}(\mathbb{D})}\|g\|_{L^{p}\left(\mathbb{R}^{d+1}\right)}, \\
&\left\|f^{\varepsilon} \partial_{t} S_{\varepsilon}(g)\right\|_{L^{p}\left(\mathbb{R}^{d+1}\right)} \leq C \varepsilon^{-2}\|f\|_{L^{p}(\mathbb{D)}}\|g\|_{{L^{p}}\left(\mathbb{R}^{d+1}\right)},
\end{aligned}\end{equation}
and
\begin{equation}\begin{aligned}
&\left\|f_1^{\varepsilon} \nabla S_{\varepsilon}(g)\right\|_{L^{p}\left(\mathbb{R}^{d+1}\right)} \leq C \varepsilon^{-1}\|f_1-f_2\|_{L^{p}(\mathbb{D})}\|g\|_{L^{p}\left(\mathbb{R}^{d+1}\right)}+C \varepsilon^{-1}\|f_2\|_{L^{p}(\mathbb{T}^{d+1})}\|g\|_{L^{p}\left(\mathbb{R}^{d+1}\right)}, \\
&\left\|f_1^{\varepsilon} \partial_{t} S_{\varepsilon}(g)\right\|_{L^{p}\left(\mathbb{R}^{d+1}\right)} \leq C \varepsilon^{-2}\|f_1-f_2\|_{L^{p}(\mathbb{D})}\|g\|_{L^{p}\left(\mathbb{R}^{d+1}\right)}+C \varepsilon^{-2}\|f_2\|_{L^{p}(\mathbb{T}^{d+1})}\|g\|_{L^{p}\left(\mathbb{R}^{d+1}\right)},
\end{aligned}\end{equation}
for any $1\leq p<\infty$, where $C$ depends only on $d$ and $p$.
\end{rmk}

\subsection{Error estimates in $L^{\frac{2(d+2)}{d}}(\Omega_T)$}
Let $\eta(t)$ be a smoothing cut-off function such that $0\leq \eta\leq 1$, $\eta=1$ if $3T/2\geq t\geq 4\varepsilon^2 $; $\eta=0$ if $t\leq 3\varepsilon^2$ or $t\geq 3T/2+4\varepsilon^2$, with $|\eta'(t)|\leq C\varepsilon^{-2}$.

Moreover, let
\begin{equation}
\begin{aligned}
w_{\varepsilon}(x, t)=u_{\varepsilon}(x, t)-u_{0}(x, t) &-\varepsilon \chi_{j}\left(x / \varepsilon, t / \varepsilon^{2}\right) S_{\varepsilon}\left(\eta\tilde{U}_{0,j}\right) \\
&-\varepsilon^{2} \phi_{(d+1) i j}\left(x / \varepsilon, t / \varepsilon^{2}\right) \frac{\partial}{\partial x_{i}} S_{\varepsilon}\left(\eta\tilde{U}_{0,j}\right).
\end{aligned}
\end{equation}

\begin{thm}Assume $A(y,s)$ satisfies the conditions $(1.2)$-$(1.4)$, and let $u_\varepsilon$ and $u_0$ be the solutions of Equations $(1.1)$ and $(1.7)$, respectively.
Let $w_\varepsilon$ be defined by $(3.14)$ with $\tilde{U}_0(x,t)=:(\nabla P(x))^{-1}\cdot \nabla u_0(x,t)$ for $0<t\leq T$, and $\tilde{U}_0(x,t)=:(\nabla P(x))^{-1}\cdot \nabla u_0(x,2T-t)$ for $T<t< 2T$. Then for any $\varphi\in L^2(0,T;\dot{H}^1(\mathbb{R}^d))$,
\begin{equation}
\begin{aligned}
&\int_{0}^{T}\left\langle\partial_{t} w_{\varepsilon}, \psi\right\rangle_{H^{-1}(\mathbb{R}^d) \times
\dot{H}^1(\mathbb{R}^d)}+\iint_{\Omega_{T}} A^{\varepsilon} \nabla w_{\varepsilon}\nabla \psi \\
=&\iint_{\Omega_{T}}(\widehat{A}_{ik}-A_{ik}^{\varepsilon})\partial_k P_\ell\left(\tilde{U}_{0,\ell}-S_{\varepsilon}(\eta\tilde{U}_{0,\ell})\right) \partial_i
\psi-\varepsilon \iint_{\Omega_{T}} A_{i j}^{\varepsilon}\chi_{k}^{\varepsilon}\partial_{j}
S_{\varepsilon}(\eta\tilde{U}_{0,k})\partial_i \psi \\
&-\varepsilon \iint_{\Omega_{T}} \phi_{k i j}^{\varepsilon}\partial_i
S_{\varepsilon}(\eta\tilde{U}_{0,j})\partial_k \psi
-\varepsilon^{2} \iint_{\Omega_{T}} \phi_{k(d+1) j}^{\varepsilon}\partial_{t}
S_{\varepsilon}(\eta\tilde{U}_{0,j})\partial_k \psi \\
&+\varepsilon \iint_{\Omega_{T}} A_{i j}^{\varepsilon}\left(\partial_j(\phi_{(d+1) \ell
k})\right)^{\varepsilon}\partial_\ell S_{\varepsilon}(\eta\tilde{U}_{0,k})
\partial_i \psi\\
&+\varepsilon^{2} \iint_{\Omega_{T}} A_{i j}^{\varepsilon}\phi_{(d+1) \ell k}^{\varepsilon}
\partial^{2}_{j\ell} S_{\varepsilon}(\eta\tilde{U}_{0,k}) \partial_i \psi.
\end{aligned}
\end{equation}
The repeated indices $i,j,k,\ell$ are summed from $1$ to $d$.
\end{thm}

\begin{proof}
The similar computation have been performed in \cite[Theorem 2.2]{MR3596717}, and we provide it for completeness. Note that $\tilde{U}_0(x,t)=:(\nabla P(x))^{-1}\cdot \nabla u_0(x,t)$ for $0<t\leq T$, then a direct computation shows that
\begin{equation}
\begin{aligned}
(\partial_{t}+\mathcal{L}_{\varepsilon})
w_{\varepsilon}=&(\mathcal{L}_{0}-\mathcal{L}_{\varepsilon})
u_{0}-(\partial_{t}+\mathcal{L}_{\varepsilon})\left\{\varepsilon \chi_{j}^{\varepsilon}
S_{\varepsilon}(\eta\tilde{U}_{0,j})\right\} -\left(\partial_{t}+\mathcal{L}_{\varepsilon}\right)\left\{\varepsilon^{2} \phi_{(d+1) i j}^{\varepsilon}
\partial_{i} S_{\varepsilon}(\eta\tilde{U}_{0,j})\right\} \\
=&-\partial_{i}\left\{(\widehat{A}_{ik}-A_{ik}^{\varepsilon})\partial_k
P_\ell\left(\tilde{U}_{0,\ell}-S_{\varepsilon}(\eta\tilde{U}_{0,\ell})\right)\right\} \\
&-\partial_{i}\left\{(\widehat{A}_{ik}-A_{ik}^{\varepsilon})\partial_k P_\ell \cdot S_{\varepsilon}(\eta\tilde{U}_{0,\ell})\right\}-\left(\partial_{t}+\mathcal{L}_{\varepsilon}\right)\left\{\varepsilon
\chi_{j}^{\varepsilon} S_{\varepsilon}(\eta\tilde{U}_{0,j})\right\} \\
&-\left(\partial_{t}+\mathcal{L}_{\varepsilon}\right)\left\{\varepsilon^{2} \phi_{(d+1) i j}^{\varepsilon}
\partial_{i} S_{\varepsilon}(\eta\tilde{U}_{0,j})\right\}.
\end{aligned}
\end{equation}
A direct computation after using the definition of $B$ in Theorem 2.8 yields that
\begin{equation}
\begin{aligned}
&-\partial_{i}\left\{(\widehat{A}_{ik}-A_{ik}^{\varepsilon})\partial_k P_\ell \cdot S_{\varepsilon}(\tilde{U}_{0,\ell})\right\}-\left(\partial_{t}+\mathcal{L}_{\varepsilon}\right)\left\{\varepsilon
\chi_{j}^{\varepsilon} S_{\varepsilon}(\eta\tilde{U}_{0,j})\right\}\\
=&\partial_{i}\left\{B_{i j}^{\varepsilon}S_{\varepsilon}(\eta\tilde{U}_{0,j})\right\}+\varepsilon \partial_{i}\left\{A_{i j}^{\varepsilon}\chi_{k}^{\varepsilon}\partial_{j} S_{\varepsilon}(\eta\tilde{U}_{0,k})\right\}
-\varepsilon \partial_{t}\left\{\chi_{j}^{\varepsilon} S_{\varepsilon}(\eta\tilde{U}_{0,j})\right\}\\
=&\varepsilon \partial_{i}\left\{A_{i j}^{\varepsilon}\chi_{k}^{\varepsilon}\partial_{j} S_{\varepsilon}(\eta\tilde{U}_{0,k})\right\}+B_{i j}^{\varepsilon}\partial_{i} S_{\varepsilon}(\eta\tilde{U}_{0,j})
-\varepsilon \chi_{j}^{\varepsilon} \partial_tS_{\varepsilon}(\eta\tilde{U}_{0,j}),
\end{aligned}
\end{equation}
where, in the inequality above, we have used $$\sum_{I=1}^{d+1}\partial_I B_{Ij}=\sum_{i=1}^{d}\partial_i B_{ij}-\partial_s \chi_j=0 \quad \text{in }\mathbb{D}$$ due to $(2.5)$.
According to $(2.33)$, we have
\begin{equation*}
\begin{aligned}
&B_{i j}^{\varepsilon}\partial_{i} S_{\varepsilon}(\eta\tilde{U}_{0,j})-\varepsilon \chi_{j}^{\varepsilon} \partial_tS_{\varepsilon}(\eta\tilde{U}_{0,j})\\
=&\varepsilon \partial_{k}\phi_{k i j}^{\varepsilon}\cdot \partial_{i} S_{\varepsilon}(\eta\tilde{U}_{0,j}) +\varepsilon^{2} \partial_{t}\phi_{(d+1) i j}^{\varepsilon} \cdot \partial_{i} S_{\varepsilon}(\eta\tilde{U}_{0,j})\\
&+\varepsilon^{2} \partial_k \phi_{k(d+1)  j}^{\varepsilon}\partial_tS_{\varepsilon}(\eta\tilde{U}_{0,j}),
\end{aligned}
\end{equation*}
where we have used the skew-symmetry of $\phi_{(d+1)(d+1)j}$. Moreover, by the skew-symmetry of $\phi$, we have
\begin{equation}
\begin{aligned}
&B_{i j}^{\varepsilon}\partial_{i} S_{\varepsilon}(\eta\tilde{U}_{0,j})-\varepsilon \chi_{j}^{\varepsilon} \partial_tS_{\varepsilon}(\eta\tilde{U}_{0,j})\\
=&\varepsilon \partial_{k}\left(\phi_{k i j}^{\varepsilon}\cdot \partial_{i} S_{\varepsilon}(\eta\tilde{U}_{0,j})\right) +\varepsilon^{2} \partial_{t}\left(\phi_{(d+1) i j}^{\varepsilon} \cdot \partial_{i} S_{\varepsilon}(\eta\tilde{U}_{0,j})\right)\\
&+\varepsilon^{2} \partial_k \left(\phi_{k(d+1)  j}^{\varepsilon}\partial_tS_{\varepsilon}(\eta\tilde{U}_{0,j})\right),
\end{aligned}
\end{equation}
which, combined with $(3.16)$-$(3.17)$, yields the desired equality $(3.15)$.
\end{proof}

Now, we are ready to give the proof of Theorem 1.2.\\

\noindent \textbf{Proof of Theorem 1.2}.
First, due to $(2.9)$, $(2.32)$, Theorem 2.9 and Remark 3.4, we know that $w_\varepsilon\in L^2(0,T;\dot{H}^1(\mathbb{R})^d)$. Let $\psi=w_\varepsilon$ in $(3.15)$, then
 Theorem 3.5 yields that the l.h.s. of $(3.15)$ is bounded by
\begin{equation}
\begin{aligned}
&C \iint_{\Omega_{T}}\left|\tilde{U}_{0}-S_{\varepsilon}(\eta\tilde{U}_{0})\right||\nabla w_\varepsilon| \\
&\quad +C \varepsilon \iint_{\Omega_{T}}\left\{\left|\chi^{\varepsilon}\right|+\left|\phi^{\varepsilon}\right|+\left|(\nabla \phi)^{\varepsilon}\right|\right\}\left|\nabla S_{\varepsilon}(\eta\tilde{U}_{0})\right||\nabla w_\varepsilon| \\
&\quad +C \varepsilon^{2} \iint_{\Omega_{T}}\left|\phi^{\varepsilon}\right|\left\{\left|\partial_{t} S_{\varepsilon}(\eta\tilde{U}_{0})\right|+\left|\nabla^{2} S_{\varepsilon}(\eta\tilde{U}_{0})\right|\right\}|\nabla w_\varepsilon| \\
&=:M_{1}+M_{2}+M_{3},
\end{aligned}
\end{equation}
where $C$ depends only on $d$ and $\kappa$.

 To estimate $M_1$, after noting the choice of $\eta$, we have
\begin{equation}
\begin{aligned}
M_{1} \leq & C \iint_{\Omega_{5\varepsilon^2}}\left\{|\tilde{U}_{0}|+S_{\varepsilon}\left(\eta |\tilde{U}_{0}|\right)\right\}|\nabla w_\varepsilon|+C \iint_{\Omega_{T} \backslash \Omega_{5 \varepsilon^2}}\left|\tilde{U}_{0}-S_{\varepsilon}(\tilde{U}_{0})\right||\nabla w_\varepsilon| \\
\leq &C\left(\iint_{\Omega_{ 6\varepsilon^2}}|\tilde{U}_{0}|^{2}\right)^{1 / 2}
||\nabla w_\varepsilon||_{L^2(\Omega_T)}+C\left\|\tilde{U}_{0}-S_{\varepsilon}(\tilde{U}_{0})\right\|_{L^2(\Omega_T \backslash \Omega_{5 \varepsilon^2})}||\nabla w_\varepsilon||_{L^2(\Omega_T)}\\
\leq & C\left(\iint_{\Omega_{ 6\varepsilon^2}}|\nabla {u}_{0}|^{2}\right)^{1 / 2}
||\nabla w_\varepsilon||_{L^2(\Omega_T)}+C\left\|\tilde{U}_{0}-S_{\varepsilon}(\tilde{U}_{0})\right\|_{L^2(\Omega_T \backslash \Omega_{5 \varepsilon^2})}||\nabla w_\varepsilon||_{L^2(\Omega_T)}.
\end{aligned}
\end{equation}

To estimate the second term in the r.h.s. of $(3.20)$,
we let
\begin{equation}U_1(z,t)=\left\{\begin{aligned}
&\tilde{U}_0(z,t),&\quad &\text{ for }4\varepsilon^2\leq t\leq 2T-8\varepsilon^2,\\
&\tilde{U}_0(z,-t+8\varepsilon^2),&\quad &\text{ for }-T\leq t\leq 4\varepsilon^2,
\end{aligned}\right.
\end{equation}
and choose a smooth cut-off function $\eta_1(t)\in[0,1]$ satisfying $\eta_1=1$ if $-T/3\leq t \leq 4T/3$, and $\eta_1=0$ if $ t \leq -T/2$ or $t\geq 3T/2$ with $|\eta_1'|\leq CT^{-1}$.
Then, it is easy to check that $\eta_1 U_1=\tilde{U}_0$ for $4\varepsilon^2\leq t\leq 4T/3$. According to $(3.4)$-$(3.6)$, we have
\begin{equation}\begin{aligned}
&\left\|\tilde{U}_{0}-S_{\varepsilon}(\tilde{U}_{0})\right\|_{L^2(\Omega_T \backslash \Omega_{5 \varepsilon^2})}=\left\|\eta_1{U}_{1}-S_{\varepsilon}(\eta_1{U}_{1})\right\|_{L^2(\Omega_T \backslash \Omega_{5 \varepsilon^2})}\\
\leq& \left\|{U}_{1}-S^x_{\varepsilon}({U}_{1})\right\|_{L^2(\Omega_T \backslash \Omega_{5 \varepsilon^2})}+\left\|S_{\varepsilon}(\eta_1{U}_{1})-S^x_{\varepsilon}(\eta_1{U}_{1})\right\|_{L^2(\Omega_T \backslash \Omega_{5 \varepsilon^2})}\\[5pt]
\leq& C \varepsilon ||\nabla {U}_{1}||_{_{L^2(\Omega_T \backslash \Omega_{5 \varepsilon^2})}}+\left\|S_\varepsilon^x(\eta_1{U}_{1})-S_{\varepsilon}(\eta_1{U}_{1})\right\|_{L^2(\mathbb{R}^{d+1})}\\
\leq &C \varepsilon ||\nabla {U}_{1}||_{{L^2(\Omega_T \backslash \Omega_{ 5\varepsilon^2})}}+C\varepsilon^2||\eta_1\partial_t {U}_{1}||_{L^2(\mathbb{R}^{d+1})}+C\varepsilon^2||{U}_{1}\partial_t\eta_1||_{L^2(\mathbb{R}^{d+1})}\\[5pt]
\leq& C \varepsilon ||\nabla \tilde{U}_{0}||_{L^2(\mathbb{R}^d\times (4\varepsilon^2,T))}+C\varepsilon^2||\partial_t \tilde{U}_{0}||_{L^2(\mathbb{R}^d\times (4\varepsilon^2,T))}+C\varepsilon^2||\tilde{U}_{0}||_{L^2(\mathbb{R}^d\times (4\varepsilon^2,T))}\\[5pt]
\leq& C \varepsilon||f||_{L^2(\Omega_T)\cap L^\infty(0,T;L^2(\mathbb{R}^d))}+C \varepsilon^2||f_t||_{L^2(\Omega_T)},
\end{aligned}\end{equation}
where we have used the estimates $(2.39)$-$(2.40)$ in the last inequality due to $\tilde{U}_0(x,t)=:(\nabla P(x))^{-1}\cdot \nabla u_0(x,t)$ for $0<t\leq T$, and $\tilde{U}_0(x,t)=:(\nabla P(x))^{-1}\cdot \nabla u_0(x,2T-t)$ for $T<t\leq 2T$, together with the definition $(3.21)$. Therefore, combining $(2.45)$, $(3.20)$ and $(3.22)$ yields that
\begin{equation}
M_1\leq C \left(\varepsilon||f||_{L^2(\Omega_T)\cap L^\infty(0,T;L^2(\mathbb{R}^d))}+C \varepsilon^2||f_t||_{L^2(\Omega_T)}\right)||\nabla w_\varepsilon||_{L^2(\Omega_T)}.
\end{equation}

Denote $M_2=:M_{21}+M_{22}+M_{23}$. Noting that $\chi\in L^\infty(\mathbb{D})$ due to $(2.16)$, $(2.18)$, $(2.19)$ and $(2.22)$, we have
\begin{equation}\begin{aligned}
M_{21}:=&C\varepsilon \iint_{\Omega_{T}}|\chi^{\varepsilon}|\left|\nabla S_{\varepsilon}(\eta\tilde{U}_{0})\right||\nabla w_\varepsilon|\\
\leq& C\varepsilon ||\nabla w_\varepsilon||_{L^2(\Omega_T)}||\nabla \tilde{U}_{0}||_{L^2(\Omega_T )}\\[5pt]
\leq& C \varepsilon ||\nabla w_\varepsilon||_{L^2(\Omega_T)}||f||_{L^2(\Omega_T )},
\end{aligned}\end{equation}
where we have used $(2.39)$ in the last inequality since $\nabla_z \tilde{u}_0(z,t)=\tilde{U}_0(P^{-1}(z),t)$ (note that $\eta\tilde{U}_{0}$ is indeed defined in the whole space $\mathbb{R}^{d+1}$).

According to $(2.32)$ and $(3.11)$, we have
\begin{equation}\begin{aligned}
M_{22}:=&C\varepsilon \iint_{\Omega_{T}}|\phi^{\varepsilon}|\left|\nabla S_{\varepsilon}(\eta\tilde{U}_{0})\right||\nabla w_\varepsilon|\\
\leq &C\varepsilon||\nabla \tilde{U}_{0}||_{L^2(\Omega_T)}||\nabla w_{\varepsilon}||_{L^2(\Omega_T)}\\[5pt]
\leq &C \varepsilon||f||_{L^2(\Omega_T)}||\nabla w_{\varepsilon}||_{L^2(\Omega_T)}.\\[5pt]
\end{aligned}\end{equation}

Similarly, according to $(2.32)$ and $(3.11)$ again, we have
\begin{equation}\begin{aligned}
M_{23}=: &C \varepsilon \iint_{\Omega_{T}}\left|(\nabla \phi)^{\varepsilon}\right|\left|\nabla S_{\varepsilon}(\eta\tilde{U}_{0})\right||\nabla w_\varepsilon|\\
\leq &C\varepsilon||\nabla \tilde{U}_{0}||_{L^2(\Omega_T)}||\nabla w_{\varepsilon}||_{L^2(\Omega_T)}\\[5pt]
\leq &C \varepsilon||f||_{L^2(\Omega_T)}||\nabla w_{\varepsilon}||_{L^2(\Omega_T)}.\\[5pt]
\end{aligned}\end{equation}

Thus, combining $(3.24)$-$(3.26)$ yields that
\begin{equation}
M_2\leq C\varepsilon ||f||_{L^2 (\Omega_T)}||\nabla w_{\varepsilon}||_{L^2(\Omega_T)}.
\end{equation}

To proceed, denote $M_3=:M_{31}+M_{32}$. Using  $(2.32)$ and $(3.11)$ yields that
\begin{equation}\begin{aligned}
M_{31}&=:C \varepsilon^{2} \iint_{\Omega_{T}}\left|\phi^{\varepsilon}\right|\left|\partial_{t} S_{\varepsilon}(\eta\tilde{U}_{0})\right||\nabla w_\varepsilon|\\
&\leq C \varepsilon^{2}||\nabla w_{\varepsilon}||_{L^2(\Omega_T)}\left(||\eta \partial_t \tilde{U}_{0}||_{L^2(\Omega_T)}+ ||\partial_t\eta \tilde{U}_{0}||_{L^2(\Omega_T)}\right)\\[5pt]
&\leq C \left(\varepsilon||f||_{L^2(\Omega_T)\cap L^\infty(0,T;L^2(\mathbb{R}^d))}+C \varepsilon^2||f_t||_{L^2(\Omega_T)}\right)||\nabla w_\varepsilon||_{L^2(\Omega_T)},
\end{aligned}\end{equation}
where we have also used $(2.40)$ and $(2.45)$ in the last inequality. Similarly, for $M_{32}$, according to $(2.32)$ and $(3.13)$, we have
\begin{equation}\begin{aligned}
M_{32}&=:C \varepsilon^{2} \iint_{\Omega_{T}}\left|\phi^{\varepsilon}\right|\left|\nabla^{2} S_{\varepsilon}(\eta\tilde{U}_{0})\right||\nabla w_\varepsilon|\\
&\leq C\varepsilon ||\nabla \tilde{U}_{0}||_{L^2(\Omega_T)}||\nabla w_{\varepsilon}||_{L^2(\Omega_T)}\\[5pt]
&\leq C \varepsilon||f||_{L^2(\Omega_T)}||\nabla w_{\varepsilon}||_{L^2(\Omega_T)}.
\end{aligned}\end{equation}
Thus, combining $(3.23)$, $(3.27)$-$(3.29)$ yields that
\begin{equation*}\begin{aligned}
||w_\varepsilon||_{L^\infty(0,T;L^2(\mathbb{R}^d))}+||\nabla w_\varepsilon||_{L^2(\Omega_T)}
\leq C\varepsilon ||f||_{L^2(\Omega_T)\cap L^\infty(0,T;L^2(\mathbb{R}^d))}+C\varepsilon^2||f_t||_{L^2(\Omega_T)}.
\end{aligned}\end{equation*}
\noindent
By embedding inequality, we have
\begin{equation*}\begin{aligned}
||w_\varepsilon||_{L^{\frac{2(d+2)}{d}}(\Omega_T)}
\leq C\varepsilon ||f||_{L^2(\Omega_T)\cap L^\infty (0,T;L^2(\mathbb{R}^d))}+C\varepsilon^2||f_t||_{L^2(\Omega_T)}.
\end{aligned}\end{equation*}
\noindent
On the other hand, due to embedding inequality \cite[Chapter 2, Theorem 2.3]{ChenYazhe}, $(2.39)$ and $\nabla_z \tilde{u}_0(z,t)=\tilde{U}_0(P^{-1}(z),t)$ for $0<t \leq T$ as well as $\nabla_z \tilde{u}_0(z,t)=\tilde{U}_0(P^{-1}(z),2T-t)$ for $T<t <2 T$, we have
$$\begin{aligned}
&\left\|\varepsilon S_{\varepsilon}(\eta\tilde{U}_{0,j})-\varepsilon^{2}\partial_{i} S_{\varepsilon}(\eta\tilde{U}_{0,j})\right\|_{L^{\frac{2(d+2)}{d}}(\Omega_T)}\\
\leq& C \varepsilon ||\eta\tilde{U}||_{ L^{\frac{2(d+2)}{d}}(\Omega_T)}
\leq  C \varepsilon ||\nabla\tilde{u}_0||_{ L^{\frac{2(d+2)}{d}}(\Omega_T)}\\[5pt]
\leq& C \varepsilon \|\tilde{u}_0\|_{W^{2,1}_2(\Omega_T)}\leq C\varepsilon \|f\|_{L^2(\Omega_T)}.
\end{aligned}$$
In conclusion, we have obtained
\begin{equation}\begin{aligned}
||u_\varepsilon-u_0||_{L^{\frac{2(d+2)}{d}}(\Omega_T)}\leq C\varepsilon ||f||_{L^2(\Omega_T)\cap L^\infty (0,T;L^2(\mathbb{R}^d))}+C\varepsilon^2||f_t||_{L^2(\Omega_T)}.
\end{aligned}\end{equation}
Thus we complete the proof of Theorem 1.2.
\qed

\begin{rmk}
A direct computation shows that
$$\begin{aligned}
S_\varepsilon^x(\partial_t\tilde{U}_{0})(x,t)=&\int_{\mathbb{R}^d}\rho_{1,\varepsilon}(x-y)\partial_t\tilde{U}_{0}(y,t)dy\\
=&\int_{\mathbb{R}^d}\rho_{1,\varepsilon}(x-P^{-1}(z))\partial_t\tilde{U}_{0}(P^{-1}(z),t)\left|\nabla P(P^{-1}(z))\right|^{-1}dz\\
=&\int_{\mathbb{R}^d}\rho_{1,\varepsilon}(x-P^{-1}(z))\nabla_z\partial_t\tilde{u}_{0}(z,t)\left|\nabla P(P^{-1}(z))\right|^{-1}dz.
\end{aligned}$$

Note that $\left|\nabla P(P^{-1}(z))\right|^{-1}$ is piecewise constant across the interface $\mathcal{I}=:\left\{z:z_1=0\right\}$, then integrating by parts yields that
\begin{equation}\begin{aligned}
\left|S_\varepsilon^x(\partial_t\tilde{U}_{0})(x,t)\right|\leq C& \varepsilon^{-1} \int_{\mathbb{R}^d}\left|\nabla_z \rho_{1,\varepsilon}(x-P^{-1}(z))\right|\cdot |\partial_t\tilde{u}_{0}|dz\\
&\quad +C\int_{z_1=0}\rho_{1,\varepsilon}(x-P^{-1}(z))|\partial_t\tilde{u}_{0}|dz'.
\end{aligned}\end{equation}

In order to estimate the second term in the r.h.s. of $(3.31)$, one needs more regularity assumption on $\partial_t\tilde{u}_{0}$ due to Trace Theorem. Consequently, following the ideas in $(3.4)$-$(3.6)$, we may not have
\begin{equation*}\left\|\tilde{U}_{0}-S_{\varepsilon}(\tilde{U}_{0})\right\|_{L^2(\Omega_T \backslash \Omega_{5\varepsilon^2})}\leq C \varepsilon \left(||\nabla^2 \tilde{u}_0||_{L^2(\Omega_T)}+||\partial_t \tilde{u}_0||_{L^2(\Omega_T)}\right),\end{equation*}
and
\begin{equation*}
\varepsilon\left\|\partial_t S_{\varepsilon}(\tilde{U}_{0})\right\|_{L^2(\Omega_T \backslash \Omega_{ 5\varepsilon^2})}\leq C ||\partial_t \tilde{u}_0||_{L^2(\Omega_T)},
\end{equation*}
even if we have $\partial_{z_j}\tilde{u}_0=\tilde{U}_{0,j}(P^{-1}(z),t)$.

However, these regularity assumptions on the source term $f$ may be technical and not optimal. And it is an interesting problem to release these regularity assumptions on $f$, which is left for further with interest.
\end{rmk}

\section{Interior Lipschitz estimates}
In homogenization theory without an interface, we take advantage of the uniform H-convergence of the multi-scale problem $\partial_tu_\varepsilon+\mathcal{L}_\varepsilon u_\varepsilon=f$ in $Q$ to the homogeneous effective problem $\partial_t u_0+\mathcal{L}_0 u_0=f$ in $Q$, which states that the multi-scale solution $u_\varepsilon$  inherits the medium-scale regularity of the solution $u_0$, where $u_0$ satisfies  $\partial_t u_0-\operatorname{div}(A^0\nabla u_0)=f$ in $Q$ with $A^0$ constant matrix satisfying the ellipticity condition $(1.3)$.

Now, back to our case, in order to apply the compactness argument, we need first to find a better regularity estimate for $u_0$, which is stated in the following theorem:

\begin{thm}Let $\widehat{A}$ be a matrix function satisfying $(1.8)$. Let $(x_0,t_0)\in \mathbb{R}^{d+1}$, and assume that $u_0\in L^2(t_0-1,t_0;H^1(B(x_0,1)))$ is a weak solution to the following equation
\begin{equation}
\partial_t u_0(x,t)-\operatorname{div}(\widehat{A}(x)\nabla u_0(x,t))=0\quad \text{in}\quad Q_1(x_0,t_0).\end{equation}

Then, there exists a constant $C$ depending only on $d$ and $\kappa$ such that, for any $\theta\in (0,1/4)$, there holds
\begin{equation}\begin{aligned}
\sup_{(x,t)\in Q_\theta(x_0,t_0)}&\left|u_0(x,t)-u_0(x_0,t_0)-(P(x)-P(x_0))\fint_{Q_\theta(x_0,t_0)}(\nabla P)^{-1}\cdot \nabla u_0\right|\\
&\leq C \theta^{2-}\left(\fint_{Q_\theta(x_0,t_0)}|u_0|^2\right)^{1/2},
\end{aligned}\end{equation}
where the notation $\theta^{2-}$ means that $\theta^{2-}=:C_\delta \theta^{2-\delta}$ for any $\delta\in (0,1)$ with $C_\delta\rightarrow +\infty$ as $\delta\rightarrow 0$.
\end{thm}
The above formula is a generalization of the initialization step in \cite[Theorem 5.1]{MR3361284} and plays an important role in obtaining the interior Lipschitz estimates in Theorem 1.4.
\begin{proof}
First, due to De Giorgi Nash-Moser estimates \cite{MR100158}, $u_0\in C_{\text{loc}}^{\alpha,\alpha/2}(Q_1(x_0,t_0))$ for some $\alpha>0$. Since $\widehat{A}(x)$ is piecewise constant, then for any $0<r<1$, it follows from \cite[Theorem 1.4]{MR3714566} that
\begin{equation*}
|\nabla_{y,s}u_0|\leq ||u_0||_{L^2(Q_1(x_0,t_0))},\quad \text{for any }(x,t)\in (B(x_0,r)\setminus \mathcal{I})\times (t_0-r^2,t_0).
\end{equation*}

Consequently, the above estimate can be improved as
\begin{equation}
\sup_{Q_{r}(x_0,t_0)}|\nabla_{x,t}u_0|\leq C||u_0||_{L^2(Q_1(x_0,t_0))}.
\end{equation}

To proceed, the key ingredient of the proof of Theorem 4.1 is that the function $\tilde{u}_0(z,t)=:u_0(P^{-1}(z),t)$ enjoys a better regularity estimate.
Indeed, by the same argument establishing $(2.42)$ above, $\tilde{u}_0$ satisfies
\begin{equation}\partial_t\tilde{u}_0(z,t)-\tilde{A}_{ij}(z)\partial_{ij}\tilde{u}(z,t)=0\quad \text{in}\quad  P^{-1}(B(x_0,1))\times (t_0-1,t_0).\end{equation}

Then, noting $\partial_{z_j}\tilde{u}_0\in L^\infty_{\text{loc}}{(P^{-1}(B(x_0,1))\times (t_0-1,t_0))}$ and applying \cite[Theorem 2.2]{MR2300337} after multiplying a suitable cut-off function, we can obtain the following interior $W^{2,1}_{p}$ estimates:
\begin{equation}
||\tilde{u}_0||_{W^{2,1}_{p,\text{loc}}(\tilde{Q}_1(x_0,t_0))}\leq C||\tilde{u}_0||_{L^2(\tilde{Q}_1(x_0,t_0))},\quad \text{for any }p\in(1,\infty),
\end{equation}where $\tilde{Q}_r(x_0,t_0)=:P^{-1}(B(x_0,r))\times (t_0-r^2,t_0)$.

Choose a smoothing domain $\tilde{\Omega}$ such that $P^{-1}(B(x_0,1/2))\subset\tilde{\Omega}\subset P^{-1}(B(x_0,2/3))$. In view of the embedding inequality \cite[Chapter 2, Theorem 3.4]{ChenYazhe}, we have
\begin{equation}
||\nabla \tilde{u}_0||_{C^{\alpha,\alpha/2}(\tilde{\Omega}\times (t_0-1/4,t_0))}\leq C||\tilde{u}_0||_{W^{2,1}_p(\tilde{Q}_{2/3}(x_0,t_0))},
\end{equation}
where $0<\alpha=1-\frac{d+2}{p}<1$.

In conclusion, there holds
\begin{equation}\begin{aligned}
||\partial_t \tilde{u}_0||_{L^\infty(\tilde{Q}_{1/2}(x_0,t_0))}+||\nabla \tilde{u}_0||_{C^{\alpha,\alpha/2}(\tilde{\Omega}\times (t_0-1/4,t_0))}\leq C||\tilde{u}_0||_{L^2(\tilde{Q}_{1}(x_0,t_0))}.
\end{aligned}\end{equation}
Therefore,
$$\begin{aligned}
&\left|\tilde{u}_0(P(x),t)-\tilde{u}_0(P(x_0),t_0)-\left(P(x)-P(x_0)\right)\fint_{Q_\theta(x_0,t_0)}\nabla \tilde{u}_0(P(z),s)\right|\\
\leq&\left|\tilde{u}_0(P(x),t_0)-\tilde{u}_0(P(x_0),t_0)-\left(P(x)-P(x_0)\right)\fint_{Q_\theta(x_0,t_0)}\nabla \tilde{u}_0(P(z),s)\right|\\
&+\left|\tilde{u}_0(P(x),t)-\tilde{u}_0(P(x),t_0)\right|\\[5pt]
\leq & C \theta^{1+\alpha}||\nabla \tilde{u}_0||_{C^{\alpha,\alpha/2}(\tilde{\Omega}\times (t_0-1/4,t_0))}+C \theta^2 ||\partial_t \tilde{u}_0||_{L^\infty(\tilde{Q}_{1/2}(x_0,t_0))}\\[5pt]
\leq & C \theta^{1+\alpha}||{u}_0||_{L^2(Q_{1}(x_0,t_0))}.\\[5pt]
\end{aligned}$$

Finally, since $\nabla \tilde{u}_0(P(x),t)=(\nabla P(x))^{-1}\cdot \nabla u_0(x,t)$, we obtain $(4.2)$.
\end{proof}

\noindent\textbf{Proof of Theorem 1.4}. After obtaining the result above, we can take advantage of the
uniform H-convergence of the multi-scale problem $\partial_tu_\varepsilon+\mathcal{L}_\varepsilon
u_\varepsilon=0$ in $Q$ to the homogeneous effective problem $\partial_t u_0+\mathcal{L}_0 u_0=0$ in $Q$, and hence the multi-scale solution $u_\varepsilon$  inherits the medium-scale regularity of the solution $u_0$,
which states that $u_\varepsilon$ satisfies an estimate similar to $(4.2)$ for $\varepsilon$ small, due to
the $H$-convergence in Theorem 1.1. Next, we need to iterate the above estimate satisfied for small
$\varepsilon$  to improve the medium-scale regularity, and complete the proof via a blow-up argument. (Across the interface, we apply De Giorgi Nash-Moser estimates \cite{MR100158} and
 \cite[Theorem 1.4]{MR3714566} to ensure the interior Lipschitz estimates for $\varepsilon$=1, similar to the explanation of $(4.3)$.)

Since all of the operations above are totally similar to \cite[Theorem 4.1]{MR3974127} and \cite[Theorem 5.1]{MR3361284}, we omit the details for simplicity and refer to \cite[Theorem 4.1]{MR3974127} and \cite[Theorem 1.2]{MR3361284} for more details.

In conclusion, we have obtained the uniform interior Lipschitz estimates for the homogenous equations $\partial_tu_\varepsilon+\mathcal{L}_\varepsilon
u_\varepsilon=0$ in $Q$. Moreover, due to De Giorgi-Nash-Moser Theorem \cite{MR100158},  the operator $\partial_t+\mathcal{L}_\varepsilon$ admits a fundamental solution $\Gamma_\varepsilon(x,y;t,s)$ and satisfies the Gaussian estimate
$$|\Gamma_\varepsilon(x,y;t,s)|\leq \frac{C}{(t-s)^{d/2}}\exp\left\{-\frac{\mu|x-y|^2}{t-s}\right\}$$
for any $x,y\in \mathbb{R}^d$ and $-\infty<s<t<\infty$, where $\mu>0$ depends only on $\kappa$ and $C>0$ depends on $d$ and $\kappa$ (see also \cite{MR217444,MR855753} for lower bounds). Then, the interior Lipschitz estimates above yield the gradient estimates of $\Gamma_\varepsilon(x,y;t,s)$:
$$|\nabla_x\Gamma_\varepsilon(x,y;t,s)|+|\nabla_y\Gamma_\varepsilon(x,y;t,s)|\leq \frac{C}{(t-s)^{(d+1)/2}}\exp\left\{-\frac{\mu|x-y|^2}{t-s}\right\}.$$

Consequently, via a suitable localization and the representation of the solution with the help of the fundamental solution $\Gamma_\varepsilon(x,y;t,s)$, a direct computation would yield the desired estimate $(1.10)$ after in view of \cite[Lemma 2.2, 2.3]{MR3361284}. See the proof of \cite[Theorem 1.2]{MR3361284} for the details, which we omit for simplicity.
\qed

\begin{center}\section*{Acknowledgement}\end{center}
%

\normalem\bibliographystyle{plain}{}

\end{document}